\documentclass[12pt]{amsart}
\usepackage[dvips]{color}
\usepackage{amsmath}
\usepackage{amsxtra}
\usepackage{amscd}
\usepackage{amsthm}
\usepackage{amsfonts}
\usepackage{amssymb}
\usepackage{eucal}
\usepackage{epsfig}
\usepackage{graphics}
\usepackage{accents}

\usepackage[all]{xy}

\usepackage{tikz}
\usepackage{here} 
\usetikzlibrary{decorations.markings,shapes,matrix,calc}
\usetikzlibrary{arrows}

\numberwithin{equation}{section}
\newtheorem{thm}{Theorem}[section]
\newtheorem{prop}[thm]{Proposition}
\newtheorem{lem}[thm]{Lemma}
\newtheorem{rem}[thm]{Remark}

\newcommand{\nn}{\nonumber}

\newcommand{\ket}[1]{{| #1 \rangle}}

\newcommand{\one}{\mathbf{1}}
\newcommand{\C}{{\mathbb C}}
\newcommand{\Z}{{\mathbb Z}}

\newcommand{\E}{{\mathcal E}}
\newcommand{\F}{\mathcal F}

\newcommand{\cR}{\mathcal{R}}

\newcommand{\bc}{\check{b}}
\newcommand{\dc}{\check{d}}
\newcommand{\pc}{\check{p}}
\newcommand{\qc}{\check{q}}
\newcommand{\uc}{\check{u}}

\newcommand{\Ac}{\check{A}}
\newcommand{\Bc}{\check{B}}
\newcommand{\Ec}{\check{E}}
\newcommand{\Fc}{\check{F}}
\newcommand{\Hc}{\check{H}}
\newcommand{\Kc}{\check{K}}

\newcommand{\Uc}{\check{U}}
\newcommand{\Vc}{\check{V}}

\newcommand{\bs}{\boldsymbol}
\newcommand{\bm}{{\mathfrak{m}}}

\newcommand{\gl}{\mathfrak{gl}}

\newcommand{\e}{\epsilon}
\newcommand{\ve}{\varepsilon}

\newcommand{\mbA}{\mathbf{A}}
\newcommand{\mbB}{\mathbf{B}}
\newcommand{\mbE}{\mathbf{E}}
\newcommand{\mbF}{\mathbf{F}}
\newcommand{\mbG}{\mathbf{G}}
\newcommand{\mbK}{\mathbf{K}}

\newcommand{\bj}{{\bf{j}}}
\newcommand{\bk}{{\bf{k}}}

\newcommand{\sse}{\textsf{e}}
\newcommand{\ssec}{\check{\textsf{e}}}

\newcommand{\id}{{\rm id}}

\newcommand{\Tr}{{\rm Tr}}

\newcommand{\g}{\mathfrak{g}}

\newcommand{\U}{\mathcal{U}}

\newcommand{\Bg}{\mathfrak{B}}

\begin{document}

\begin{title}[Duality in the quantum toroidal setting]
{The $(\gl_m,\gl_n)$ duality in the quantum toroidal setting}
\end{title}

\author{B. Feigin, M. Jimbo, and E. Mukhin}
\address{BF: National Research University Higher School of Economics, 
Russian Federation, International Laboratory of Representation Theory 
and \newline Mathematical Physics, Russia, Moscow,  101000,  
Myasnitskaya ul., 20 and Landau Institute for Theoretical Physics,
Russia, Chernogolovka, 142432, pr.Akademika Semenova, 1a
}
\email{bfeigin@gmail.com}
\address{MJ: Department of Mathematics,
Rikkyo University, Toshima-ku, Tokyo 171-8501, Japan}
\email{jimbomm@rikkyo.ac.jp}
\address{EM: Department of Mathematics,
Indiana University-Purdue University-Indianapolis,
402 N.Blackford St., LD 270,
Indianapolis, IN 46202, USA}\email{emukhin@iupui.edu}

\begin{abstract}
On a Fock space constructed from $mn$ free bosons and lattice $\Z^{mn}$, 
we give a level $n$ action of the quantum toroidal algebra $\E_m$ 
associated to $\gl_m$, together with a level $m$ action of the quantum 
toroidal algebra $\E_n$ associated to $\gl_n$.
We prove that the $\E_m$ transfer matrices commute with the $\E_n$ 
transfer matrices after an appropriate identification of parameters. 
\end{abstract}

\date{\today}

\maketitle

\section{Introduction}\label{sec:intro}
Duality of integrable systems is a very interesting, deep, and somewhat mysterious 
phenomenon which is being intensively explored.

\smallskip

The simplest example is the $(\gl_m,\gl_n)$ duality of Gaudin systems. 
Recall that for any reductive Lie algebra $\g$, and any $\g$ weight 
$\bs p\in\mathfrak{h}^*$, there exists a remarkable commutative 
subalgebra $\Bg(\bs p)\subset U(\g[x])$ called the algebra of 
higher Gaudin Hamiltonians. Here $\g[x]=\g\otimes \C[x]$ is 
the current algebra, and $\mathfrak{h}\subset \g$ is the Cartan subalegbra. 
The algebra $\Bg(\bs p)$ can be obtained
using structures of conformal field theory and the so-called 
``center on the critical level'', see \cite{FFR}.

Let $V_k$ be the vector representation of $\gl_k$. Then, by the classical 
construction, we have commuting actions of $\gl_m$ and $\gl_n$ in the spaces 
$S^*(V_m\otimes V_n)$ and $\wedge^* (V_m\otimes V_n)$. These actions can be 
extended to the actions of $\gl_m[x]$ and $\gl_n[x]$ such that 
$V_m\otimes V_n=\oplus_{i=1}^n V_m(u_i)$ as $\gl_m[x]$ module and 
$V_m\otimes V_n=\oplus_{j=1}^m V_n(\uc_j)$ as $\gl_n[x]$ module. 
Here $u_i, \uc_j\in\C$ are arbitrary parameters, and $V_m(u_i)$, 
$V_n(\uc_j)$ are $\gl_m[x]$ and $\gl_n[x]$  evaluation modules, respectively.
The actions of $\gl_m[x]$ and $\gl_n[x]$ do not commute.

The parameters $\{u_i\}$ canonically determine a $\gl_n$ weight 
$\check{\bs p}\in(\mathfrak{h}_n)^*$, while
the parameters $\{\uc_j\}$ canonically determine a $\gl_m$ weight 
$\bs p\in (\mathfrak{h}_m)^*$. It turns out that the two commutative 
subalgebras of higher Gaudin Hamiltonians $\Bg(\bs p)\subset U(\gl_m[x])$ 
and $\Bg(\check{\bs p})\subset U(\gl_n[x])$ commute with each other, see \cite{MTV1}. 
Moreover, these subalgebras coincide and one can write explicitly 
$\gl_m$ higher Gaudin Hamiltonians in terms of 
$\gl_n$ higher Gaudin Hamiltonians and vice versa. 

The spectrum of the Gaudin model is found by the Bethe ansatz method. 
The duality described above gives a correspondence between solutions 
of the Bethe ansatz equations for $\gl_m$ and $\gl_n$ models 
which can also be described via an appropriate Fourier transform related to the bispectral involution for rational solutions of the KP hierarchy, see \cite{MTV2}.  

In \cite{MTV3}, a similar duality is described 
on the level of the Bethe ansatz equations between the
trigonometric $\gl_m$ Gaudin model and the XXX $\gl_n$ model (related to Yangians). 
It is expected that the $(\gl_m,\gl_n)$ duality can be lifted to the duality of the XXZ models
(related to quantum affine algebras). 

\smallskip

Another example, which motivated this paper, is a duality between
local and non-local integrals of motion in the quantum KdV system 
\cite{BLZ1}--\cite{BLZ3}.
Consider the 
free vertex operator algebra generated
by one field $h(z)=\sum_{i\in\Z}h_iz^{-i}$, where
$\{h_i\}$ are generators of the Heisenberg algebra satisfying
$[h_i,h_j]=i\delta_{i,-j}$. 
The local integrals of motion are given by 
integrals of the form $\int T_{2n}(z)\frac{dz}{z}$, 
$n\in\Z_{\geq 0}$, 
where $T_{2n}(z)$
are local currents. For example $T_2(z)$ is the Virasoro current 
$T_2(z)=\frac{1}{2}:h(z)^2:+\lambda h'(z)$, $\lambda\in\C$,
$T_4(z)=:T_2(z)^2:$, and so on. 

Non-local integrals of motion are multiple integrals of vertex operators.
For example, the first one has the form 
$\int\!\!\int S_1(z)S_2(w)\frac{dz}{z}\frac{dw}{w}$,
where $S_1(z),S_2(w)$ are vertex operators, 
the second non-local integral of motion is a four-fold integral, and so forth. 

In the KdV case local and non-local integrals look very different,
but they mutually commute. 
Understanding the situation with Bethe ansatz and the correspondence 
between the spectra of local and non-local integrals of motion
is still far from complete.

There are many integrable systems of KdV type. Conjecturally,
it is possible to find local and non-local integrals of motion
in almost all $W$ algebras, in the universal enveloping algebras
of affine Kac-Moody Lie algebras (including superalgebras), and in coset vertex algebras.

\smallskip

In the $(\gl_m,\gl_n)$ duality mentioned above, the two sets of integrals in duality are given by 
the same construction.
In contrast, in the case of  vertex operator algebras, the constructions of 
local and non-local integrals of motion are very different.
The situation can be explained if we consider quantum versions. 

In \cite{FKSW}, \cite{KS}, 
``local" and ``non-local" integrals of motion are constructed inside
certain deformed $W$ algebras: 
the quantum toroidal $\gl_1$ and $\gl_n$ algebras, see \cite{FJM}. 
Both are given by similar multiple integrals (hence they are both non-local). 
In the conformal limit the deformed 
$W$ algebras turn into
vertex operator algebras, and the two 
sets of integrals of motion become different looking local and non-local 
integrals of motion. Moreover, the same deformed 
$W$ algebras admit alternative conformal 
limits where the situation is reversed: ``local" integrals of motion become 
non-local and ``non-local" ones become local.
\smallskip

In this paper we study a quantum affine 
analog of the $(\gl_m,\gl_n)$ duality. 

Let $\E_k=\E_k(q_1,q_2,q_3)$ be the quantum toroidal algebra of type $\gl_k$, 
where $k\in\Z_{\geq 1}$ and $q_1, q_2,q_3\in\C$ are parameters satisfying 
$q_1q_2q_3=1$. For the definition, see Appendix \ref{sec:Em}. 
The quantum toroidal algebra has a vertical subalgebra isomorphic to 
the quantum affine algebra $U_q\widehat{\gl}_k\subset \E_k$ where $q^2=q_2$. 
Inside a completion  ${\widetilde \E}_k$ of $\E_k$ one has a commutative algebra of transfer 
matrices 
$\widehat\Bg_q(\bs p)$, 
which we call the Bethe algebra 
of integrals of motion. The Bethe algebra of integrals of motion depends 
on an affine weight $\bs p\in (\mathfrak{ \hat {h}}_k)^*$, see \cite{FJMM} and Subsection \ref{subsec:IM} below.

We consider $mn$ free bosons and a lattice $\Z^{mn}$, where 
$m,n\in\Z_{\geq 1}$.
We use the vertex operator construction of \cite{Sa} to define actions of 
$\E_m(q_1,q_2,q_3)$ of level $n$ and of $\E_n(\qc_1,q_2,\qc_3)$ of level 
$m$ on the total Fock space $\Bbb F_{m,n}$. The actions depend on
parameters $u_1,\dots, u_{n}$ and $\uc_1,\dots, \uc_{m}$, respectively\footnote{In the main text we use notation $u_0=u_n$ and $\uc_0=\uc_m$.}.

Then the actions of vertical subalgebras $U_q\widehat{\gl}_m$ 
and $U_q\widehat{\gl}_n$ commute. The total Fock space $\Bbb F_{m,n}$ 
as a $U_q\widehat{\gl}_m$ module is a direct sum of various products 
of $n$ integrable irreducible level one modules. As $\E_m(q_1,q_2,q_3)$ 
modules those products also become tensor products 
of irreducible modules with spectral parameters, 
see Lemma \ref{lem:levn}. 
Similarly, as a $U_q\widehat{\gl}_n$ module $\Bbb F_{m,n}$ is a direct 
sum of various products of $m$ integrable irreducible level one 
modules which are irreducible $\E_n(\qc_1,q_2,\qc_3)$ modules with spectral parameters. 

Next, we consider Bethe algebras. 
The new feature is that, while 
the parameters $\{\uc_j\}$ determine finite-dimensional part of 
the affine weight $\bs p\in (\mathfrak{\hat h}_m)^*$,
the parameter  $\qc_1$ determines the ``loop direction'' of the latter. 
Similarly  the dual affine weight
$\check{\bs p}\in (\mathfrak{\hat h}_n)^*$ is determined from $\{u_i\}$ and $q_1$.
Then we prove that the corresponding Bethe algebras
$\widehat\Bg_q (\bs p)\subset \widetilde{\E}_m$ and 
$\widehat\Bg_q (\check{\bs p})\subset \widetilde{\E}_n$ mutually commute,  
see Theorem \ref{thm:IM}. 

The action of integrals of motion in $\Bbb F_{m,n}$ is given
by multiple integrals described in \cite{FJM}, see also Proposition \ref{prop:IM} below.
The integrals of motion of \cite{FKSW}, \cite{KS} can be recovered 
in the case of $m=1$.

\smallskip

The spectrum of Bethe algebra of integrals of motion is expected to be given by 
Bethe ansatz method. For the case of $\E_1$ it is proved in \cite{FJMM1}, 
\cite{FJMM2}. For the case of $\E_2$ the precise statement is conjectured 
in \cite{FJM}. Therefore, the duality described in this paper suggests a 
correspondence of solutions of two different Bethe ansatz equations. 
It would be interesting to understand this correspondence. 

The KdV nonlocal and local integrals of motion are expected to be related to coefficients of 
expansions of the same function at zero and infinity respectively, see formulas (55) and (66) in \cite{BLZ1}. 
It is an important problem to find and to prove a similar relation between integrals of motion corresponding 
to $\E_n$ and $\E_m$.

Also it would be interesting to study the various limits of the duality 
described in this paper. 

From the technical point of view, 
our proof of duality is purely computational, 
following the lines of \cite{FKSW},\cite{KS}. 
To give a more conceptual explanation to the duality phenomena
 is an issue left for the future.
\smallskip

The plan of the paper is as follows. 

In Section \ref{sec:prel} we collect preliminary materials concerning
the bosons and the total Fock space $\mathbb{F}_{m,n}$. 
In Section \ref{sec:Fock} we define actions of two
quantum toroidal algebras $\E_m=\E_m(q_1,q_2,q_3)$
and $\check{\E}_n=\E_n(\qc_1.q_2,\qc_3)$ on $\mathbb{F}_{m,n}$,
and state the commutativity of the quantum affine subalgebras 
$U_q\widehat{\gl}_m$ and $U_q\widehat{\gl}_n$ 
(Theorem \ref{affine-com}).
In Section \ref{sec:duality} we introduce integrals of motion associated 
with $\E_m$ and $\check{\E}_n$. We then state the main theorem
concerning their commutativity (Theorem \ref{thm:IM}).

The text is followed by four appendices. 
In Appendix \ref{sec:Em} we give the presentation of the quantum toroidal
algebra $\E_m$. 
In Appendix \ref{sec:contraction}, contractions of various currents
are summarized in tables. 
The commutativity of the quantum affine subalgebras 
$U_q\widehat{\gl}_m$ and $U_q\widehat{\gl}_n$ 
is proved in Appendix \ref{sec:affine-com}. 
Proof of the duality is given in Appendix \ref{sec:cont}. 

\bigskip

\noindent
{\it Notation.}\quad
Throughout the text we fix 
parameters $q,d,\dc\in\C^{\times}$ and define
\begin{align*}
&q_1=q^{-1}d\,,\ q_2=q^2\,,\ q_3=q^{-1}d^{-1}\,,\\ 
&\check{q}_1=q^{-1}\check{d}\,,\ \check{q}_2=q^2\,,\
\check{q}_3=q^{-1}\check{d}^{-1}\,, 
\end{align*}
so that $q_2=\qc_2$, 
$q_1q_2q_3=\check{q}_1\check{q}_2\check{q}_3=1$. 
We assume that $(q_1,q_2,q_3)$ are generic, in the sense that if   
$q_1^iq_2^jq_3^k=1$ for $i,j,k\in\Z$, then $i=j=k$. 
We assume the same for $(\qc_1,\qc_2,\qc_3)$.

We use symbols for ordered product
\[
\prod_{1\le i\le N}^{\curvearrowright}A_i=A_1A_2\cdots A_N\,,
\quad \prod_{1\le i\le N}^{\curvearrowleft}A_i=A_NA_{N-1}\cdots A_1\,
\]
and infinite products
\begin{align*}
&(z_1,\ldots,z_r;p)_\infty=\prod_{i=1}^r\prod_{k=0}^\infty(1-z_i p^k)\,,
\quad 
\Theta_p(z)=(z,pz^{-1},p;p)_\infty\,.
\end{align*}

We set $\theta(P)=1$ if statement $P$ is true and $\theta(P)=0$ otherwise. 
For a positive integer $N$, we write
$\delta^{(N)}_{i,k}=\theta\bigl(i\equiv k\bmod N\bigr)$.
\medskip

\section{Preliminaries}\label{sec:prel}
In this section we collect preliminary materials which will be used
to construct representations of algebras $\E_m$ and $\check{\E}_n$
in Section \ref{sec:Fock}.

\subsection{Total Fock space}\label{sec:totalFock}
Let $m,n$ be positive integers. 
We introduce a set of bosons
$\{a^{i,j}_{r}\mid r\in\Z\setminus\{0\},\ 0\le i\le m-1,\ 0\le j\le n-1\}$
such that 
\begin{align}
[a^{i,j}_r,a^{k,l}_{s}]=\delta_{i,k}
\delta_{j,l}\delta_{r+s,0}\frac{[r]^2}{r}\,, 
\label{aa}
\end{align}
where $[x]=(q^x-q^{-x})/(q-q^{-1})$.

Consider a vector space with basis $\{\ket{\bm}\}$ labeled by 
$\bm=(\bm_{s,t})_{0\le s\le m-1\atop 0\le t\le n-1}\in\Z^{mn}$. 
We define linear operators $e^{\pm\e_{i,j}}$, $\partial_{i,j}$ by setting
\begin{align*}
&e^{\pm\e_{i,j}}\ket{\bm}=(-1)^{\sum_{s=0}^{i-1}\sum_{t=0}^{n-1}\bm_{s,t}
+\sum_{t=0}^{j-1}\bm_{i,t}}
\ket{\bm\pm\one_{i,j}}\,,\\
&\partial_{i,j}\ket{\bm}=\bm_{i,j}\ket{\bm}\,,
\end{align*}
where $\one_{i,j}=(\delta_{is}\delta_{jt})$. 
We have then 
\begin{align*}
&e^{\e_{i,j}} e^{\e_{k,l}} =-e^{\e_{k,l}} e^{\e_{i,j}}\quad 
\text{if $(i,j)\neq(k,l)$},
\\
&\ket{\bm}=
\prod_{0\le i\le m-1}^{\curvearrowright}\prod_{0\le j\le n-1}^{\curvearrowright}
e^{\bm_{i,j}\e_{i,j}}\ket{\mathbf{0}}
=e^{\bm_{0,0}\e_{0,0}}\cdots e^{\bm_{m-1,n-1}\e_{m-1,n-1}}\ket{\mathbf{0}}
\,.
\end{align*}

Define the total Fock space  by
\begin{align*}
\mathbb{F}_{m,n}=\C[\{a^{i,j}_{-r}\}_{r>0, 0\le i\le m-1\atop 0\le j\le n-1}] \otimes 
\Bigl(\bigoplus_{\bm\in\Z^{mn}} \C \ket{\bm}
\Bigr)\,.
\end{align*}
We write the operators $a^{i,j}_{\pm r}\otimes \id$, 
$\id\otimes e^{\e_{i,j}}$, $\id\otimes \partial_{i,j}$ on
$\mathbb{F}_{m,n}$
simply as $a^{i,j}_{\pm r}$, $e^{\e_{i,j}}$, $\partial_{i,j}$. 
We extend the range of the superfixes $i,j$ to $\Z$
by demanding periodicity $a^{i,j}_r=a^{i+m,j}_r=a^{i,j+n}_r$, 
and likewise for $e^{\e_{i,j}}$, $\partial_{i,j}$. 

The space $\mathbb{F}_{m,n}$ carries a $\Z$
grading with the degree assignment
\[
\deg (a^{i_1,j_1}_{-r_1}\cdots a^{i_N,j_N}_{-r_N}\ket{\bm})=
\sum_{l=1}^Nr_l+\frac{1}{2}\sum_{s,t}\bm_{s,t}^2\,. 
\]

For a polynomial $X$ in the oscillators $a^{i,j}_r$ ($r\neq0$) and 
the `zero mode' operators $e^{\e_{i,j}}$, $\partial_{i,j}$,  
we shall use the normal ordering symbol :$X$:. 
The rule is as usual; $a^{i,j}_r$ with $r>0$ (resp. $r<0$)
are placed to the right (resp. left), and 
 $\partial_{i,j}$ are placed to the right of $e^{\e_{k,l}}$. 
The ordering of the $e^{\e_{k,l}}$'s is kept unchanged, so that
$:e^{\e_{k_1,l_1}}\cdots e^{\e_{k_r,l_r}}:=
e^{\e_{k_1,l_1}}\cdots e^{\e_{k_r,l_r}}$. 

\medskip

\subsection{Bosons $b^{i,j}_r$}\label{boson-b}
We introduce auxiliary bosons $b^{i,j}_{\pm r}$, $\bc^{i,j}_{\pm r}$ 
by setting 
\begin{align}
&b^{i,j}_r=q^r(q_3^{r}a^{i-1,j}_r-a^{i,j}_r)\,,
\quad
b^{i,j}_{-r}=q_1^{r}a^{i-1,j}_{-r}-a^{i,j}_{-r} \,,
\label{def-b1}\\
&\bc^{i,j}_r=-q^r(\qc_3^{\,r}a^{i,j-1}_r-a^{i,j}_r)\,,
\quad
\bc^{i,j}_{-r}
=-(\qc_1^{\,r}a^{i,j-1}_{-r}-a^{i,j}_{-r})\,,
\label{def-b2}
\end{align}
for $r>0$.  
They are subject to linear relations
\begin{align}
&b^{i,j}_{r}+\bc^{i,j}_{r}
=
\qc_3^{\,r} b^{i,j-1}_{r}+q_3^r \bc^{i-1,j}_{r}
\,,
\label{bbc-rel2}\\
&b^{i,j}_{-r}+\bc^{i,j}_{- r}
=\qc_1^{\,r}b^{i,j-1}_{-r}+q_1^r\bc^{i-1,j}_{-r}\,.
\label{bbc-rel1}
\end{align}

Commutation relations of these bosons are obtained from \eqref{aa}.  
\begin{lem}
The bosons \eqref{def-b1}--\eqref{def-b2}
satisfy the commutation relations
\begin{align}
&[b^{i,j}_r,b^{k,l}_{-r}]=\frac{[r]^2}{r}
q^r\bigl((1+q_2^{-r})\delta^{(m)}_{i,k}-q_1^r\delta^{(m)}_{i+1,k}
-q_3^r\delta^{(m)}_{i-1,k}\bigr)\delta^{(n)}_{j,l} \,,
\label{bb1}\\
&[\bc^{i,j}_r,\bc^{k,l}_{-r}]=\frac{[r]^2}{r}q^r
\delta^{(m)}_{i,k}
\bigl((1+q_2^{-r})\delta^{(n)}_{j,l}-\qc_1^r\delta^{(n)}_{j,l-1}
-\qc_3^{r}\delta^{(n)}_{j,l+1}\bigr) \,,
\label{bb2}\\
&[b^{i,j}_r,\bc^{k,l}_{-r}]=-\frac{[r]^2}{r}q^{r}
\bigl(q_3^r\delta^{(m)}_{i-1,k}-\delta^{(m)}_{i,k}\bigr)
\bigl(\qc_1^r\delta^{(n)}_{j,l-1}-\delta^{(n)}_{j,l}\bigr)\,,
\label{bb3}\\
&[\bc^{k,l}_r,b^{i,j}_{-r}]=-\frac{[r]^2}{r}q^{r}
\bigl(q_1^r\delta^{(m)}_{i-1,k}-\delta^{(m)}_{i,k}\bigr)
\bigl(\qc_3^r\delta^{(n)}_{j,l-1}-\delta^{(n)}_{j,l}\bigr)\,.
\label{bb4}
\end{align}
\qed
\end{lem}

We shall also use 
currents $A^{i,j}(z)$, $B^{i,j}(z)$, $\Ac^{i,j}(z)$, $\Bc^{i,j}(z)$.
These are  formal series of the form  $X(z)=\sum_{r\neq0}X_rz^{-r}$, 
whose Fourier coefficients
$X_r$ are given in terms of $b^{i,j}_{\pm r}, \bc^{i,j}_{\pm r}$ as follows.  
\begin{align}
&A^{i,j}_r=-\frac{1}{[r]}q^{-(n-1)r}\qc_3^{-jr}b^{i,j}_r\,,
\quad 
A^{i,j}_{-r}= \frac{1}{[r]}q^{(n-2)r}
\Bigl(\qc_3^{jr}b^{i,j}_{-r}
+(1-q_2^r)\sum_{t=j+1}^{n-1}\qc_3^{tr}b^{i,t}_{-r}\Bigr)\,,
\label{Acur}\\
&B^{i,j}_{r}=\frac{1}{[r]}q^r
\Bigl(\qc_1^{jr}b^{i,j}_{r}
+(1-q_2^{-r})\sum_{t=0}^{j-1}\qc_1^{tr}b^{i,t}_r\Bigr) \,,
\quad 
B^{i,j}_{-r}=-\frac{1}{[r]}\qc_1^{-jr}b^{i,j}_{-r} \,,
\label{Bcur}\\
&\Ac^{i,j}_r=-\frac{1}{[r]}q^{-(m-1)r}
q_3^{-ir}\bc^{i,j}_r\,,
\quad 
\Ac^{i,j}_{-r}= \frac{1}{[r]}
q^{(m-2)r}
\Bigl(
q_3^{ir}\bc^{i,j}_{-r}
+(1-q_2^{r})\sum_{s=i+1}^{m-1}q_3^{sr}\bc^{s,j}_{-r}
\Bigr)\,,
\label{Accur}\\
&\Bc^{i,j}_{r}=
\frac{1}{[r]}q^r
\Bigl(q_1^{ir}\bc^{i,j}_{r}
+(1-q_2^{-r})\sum_{s=0}^{i-1}q_1^{sr}\bc^{s,j}_r 
\Bigr)\,,
\quad 
\Bc^{i,j}_{-r}=-\frac{1}{[r]}q_1^{-ir}\bc^{i,j}_{-r} \,.
\label{Bccur}
\end{align}
Here $0\le i\le m-1$, $0\le j\le n-1$ and $r>0$.  

\section{Fock representation}\label{sec:Fock}

We consider two quantum toroidal algebras in parallel:
algebra $\E_m=\E_m(q_1,q_2,q_3)$ with parameters $q_1,q_2,q_3$,
and algebra $\check{\E}_n=\check{\E}_n(\qc_1,q_2,\qc_3)$ 
with parameters $\qc_1,q_2,\qc_3$.  
Our convention about quantum toroidal algebras 
is summarized in Appendix \ref{sec:Em}. 
We say that an $\E_m$ module has level $n$ if the central element 
$C$ acts as a scalar $q^n$. 
In this section, 
we introduce a level $n$ action of $\E_m$ 
and a level $m$ action of $\check{\E}_n$
on the same total Fock space $\mathbb{F}_{m,n}$;  
see Propositions \ref{repEm} and \ref{repEn} below.

\subsection{Level one representations}\label{sec:level1}

We begin with the case $n=1$, considering a level one action of 
$\E_m$ on $\mathbb{F}_{m,1}$. 
Apart from minor modification, 
the following result is due to \cite{Sa} (see also \cite{STU}).

\begin{prop}
Let $u\in \C^\times$. 
The following formulas give a level one representation of $\E_m$ on $\mathbb{F}_{m,1}$:
\begin{align*}
&q^{\ve_i}=q^{\partial_{i,0}}\,,\quad C=q\,,\quad D=q^{\deg}\,,
\quad
H_{i,r}=b^{i,0}_r\,,\quad H_{i,-r}=b^{i,0}_{-r}\,,
\\
&E_i(z)=u^{-\delta_{i,0}}:e^{A^{i,0}(z)}:U^{i,0}(z)\,,
\quad
F_i(z)=u^{\delta_{i,0}}:e^{B^{i,0}(z)}:V^{i,0}(z)\,.
\end{align*}
Here $0\le i\le m-1$, $r>0$, 
the bosons $b^{i,0}_{\pm r}$ and the currents
$A^{i,0}(z),B^{i,0}(z)$ 
are those with $n=1$, and 
\begin{align*}
&U^{i,0}(z)=e^{-\e_{i,0}}e^{\e_{i-1,0}}
z^{\partial_{i-1,0}-\partial_{i,0}+1}
d^{(\partial_{i-1,0}+\partial_{i,0})/2}\,,
\nn\\
&V^{i,0}(z)=
e^{-\e_{i-1,0}}e^{\e_{i,0}}
z^{-\partial_{i-1,0}+\partial_{i,0}+1}
d^{-(\partial_{i-1,0}+\partial_{i,0})/2}
\,.\nn
\end{align*}
\qed
\end{prop}
We write the level one $\E_m$ module given above as $\mathbb{F}_{m,1}(u)$.
\smallskip

For $u\in\C^\times$, $\nu\in\Z/m\Z$ 
and $t,s\in\Z$, let
$\mathcal{F}_m^{(\nu)}(u;t,s)$ denote the level one irreducible 
$\E_m$ module such that 
(i) it has highest weight $(P_0(z),\ldots,P_{m-1}(z))$ 
where 
\begin{align*}
P_i(z)=1 \quad (i\neq\nu),\quad
P_\nu(z)=q \frac{1-q_2^{-1}u/z}{1-u/z}\,,
\end{align*}
(see Appendix \ref{sec:Em} for the definition of highest weight), 
(ii) the highest weight vector has degree $t$, 
and (iii) the central element 
 $q^{\sum_{i=0}^{m-1}\varepsilon_i}$ acts as a scalar $q^s$. 

\begin{lem}\label{lem:lev1}
We have a decomposition into submodules
\begin{align*}
&\mathbb{F}_{m,1}(u)=
\bigoplus_{s\in\Z}\mathcal{F}_m^{(\nu)}(u^{(s)};t^{(s)},s)\,,
\end{align*}
where $s=ml+\nu$ ($l\in\Z$, $0\le\nu\le m-1$), and
\begin{align*}
&u^{(s)}
=(-)^m d^{-s-m/2} qu\,,  \quad
t^{(s)}
=\frac{\nu}{2}(l+1)^2+\frac{m-\nu}{2}l^2\,.
\end{align*}
The submodule $\mathcal{F}_m^{(\nu)}(u^{(s)};t^{(s)},s)$ 
is generated by a highest weight vector
\[
v^{(s)} 
=\ket{\overbrace{l+1,\ldots,l+1}^{\nu},l,\ldots,l} \,.
\]
\qed
\end{lem}

\subsection{Higher level representations}\label{sec:levelmn}
For general $n\ge1$, we identify the vector space $\mathbb{F}_{m,n}$ 
with the tensor product 
\[
\mathbb{F}_{m,1}(u_0)\otimes\cdots\otimes\mathbb{F}_{m,1}(u_{n-1}) \,.
\]
We give it an $\E_m$ module structure 
via the iterated coproduct
$\Delta^{(n-1)}$ (see Appendix \ref{sec:Em} for the definition of $\Delta$).
For convenience we further gauge transform the action of $x\in\E_m$ as 
\begin{align}
&\Delta^{(n-1)} x \to d^Z(\Delta^{(n-1)} x)d^{-Z}, 
\quad Z=-\frac{1}{2}\sum_{s=0}^{m-1}\sum_{0\le t\neq t'\le n-1}
\Bigl(s+\frac{1}{2}\Bigr)\partial_{s,t}\partial_{s,t'}\,.
\label{conjug}
\end{align}
We set further
\begin{align*}
\textsf{e}_i=\sum_{t=0}^{n-1}\partial_{i,t}\,,
\quad
\check{\textsf{e}}_j=-\sum_{s=0}^{m-1}\partial_{s,j}\,.
\end{align*}

Identifying 
$1\otimes\cdots\otimes b^{i,0}_{\pm r}\otimes\cdots\otimes1$
with $d^{\pm jr}b^{i,j}_{\pm r}$,
we obtain formulas for the action of generators of $\E_m$, 
which we summarize below. 

\begin{prop}\label{repEm}
Let $m\ge2$, $n\ge1$, 
$u_0,\ldots,u_{n-1}\in\C^\times$. 
The following formulas give a level $n$ action of $\E_m$  on 
$\mathbb{F}_{m,n}$:
\begin{align}
&q^{\ve_i}=q^{\textsf{e}_i}
\,,
\quad C=q^n\,,\quad
D=q^{\deg}\,,
\label{Fock0}
\\
&H_{i,r}=\sum_{j=0}^{n-1}\qc_1^{jr}b^{i,j}_r\,,
\quad
H_{i,-r}=\sum_{j=0}^{n-1}q^{(n-1)r}\qc_3^{jr}b^{i,j}_{-r}\,,
\label{Fock1}\\
&E_i(z)=\sum_{j=0}^{n-1} u_{j}^{-\delta_{i,0}}E^{i,j}(z)\,,
\quad E^{i,j}(z)=:e^{A^{i,j}(z)}:U^{i,j}(z)\,,
\label{Fock2}\\
&F_i(z)=\sum_{j=0}^{n-1}u_{j}^{\delta_{i,0}}F^{i,j}(z)\,,
\quad F^{i,j}(z)=:e^{B^{i,j}(z)}:V^{i,j}(z)\,.
\label{Fock3}
\end{align}
where $0\le i\le m-1$, $0\le l\le n-1$ and $r>0$.
We set for $1\le i\le m-1$
\begin{align}
&U^{i,j}(z)=e^{-\e_{i,j}}e^{\e_{i-1,j}}
(q^{n-1-j}z)^{\partial_{i-1,j}-\partial_{i,j}+1}
\label{Fock4}\\
&\quad \times
d^{(1/2-i)\textsf{e}_{i-1}+(1/2+i)\textsf{e}_i+i(\partial_{i-1,j}-\partial_{i,j})}
q^{-\sum_{t=j+1}^{n-1}(\partial_{i-1,t}-\partial_{i,t})}\,,
\nn\\
&V^{i,j}(z)=
e^{-\e_{i-1,j}}e^{\e_{i,j}}
(q^{j}z)^{-\partial_{i-1,j}+\partial_{i,j}+1}
\label{Fock5}
\\
&\quad\times
d^{-(1/2-i)\textsf{e}_{i-1}-(1/2+i)\textsf{e}_i-i(\partial_{i-1,j}-\partial_{i,j})}
q^{\sum_{t=0}^{j-1}(\partial_{i-1,t}-\partial_{i,t})}
\,,\nn
\end{align}
and for $i=0$ 
\begin{align}
&U^{0,j}(z)=e^{-\e_{0,j}}e^{\e_{m-1,j}}
(q^{n-1-j}z)^{\partial_{m-1,j}-\partial_{0,j}+1}
\label{Fock6}\\
&\quad \times
d^{(1/2-m)\textsf{e}_{m-1}+(1/2)\textsf{e}_0+m\partial_{m-1,j}}
\times q^{-\sum_{t=j+1}^{n-1}(\partial_{m-1,t}-\partial_{0,t})
}
\nn\\
&V^{0,j}(z)=e^{-\e_{m-1,j}}e^{\e_{0,j}}
(q^{j}z)^{-\partial_{m-1,j}+\partial_{0,j}+1}
\label{Fock7}
\\
&\quad\times
d^{-(1/2-m)\textsf{e}_{m-1}-(1/2)\textsf{e}_0-m\partial_{m-1,j}}
\times q^{\sum_{t=0}^{j-1}(\partial_{m-1,t}-\partial_{0,t})
}
\,.\nn
\end{align}
\qed
\end{prop}
\medskip

Denote this module by 
$\mathbb{F}_{m,n}(u_0,\ldots,u_{n-1})$. 
From Lemma \ref{lem:lev1}, we obtain the following.  
\begin{lem}\label{lem:levn}
We have a decomposition into submodules
\begin{align*}
&\mathbb{F}_{m,n}
(u_0,\ldots,u_{n-1})=
\bigoplus_{s_0,\ldots,s_{n-1}\in\Z}
\mathcal{F}_m^{(\nu_0)}(
u_0^{(s_0)}
;t^{(s_0)},s_0)
\otimes\cdots\otimes 
\mathcal{F}_m^{(\nu_{n-1})}(
u_{n-1}^{(s_{n-1})}
;t^{(s_{n-1})},s_{n-1})
\,,
\end{align*}
where $s_j=ml_j+\nu_j$ ($l_j\in\Z$, $0\le\nu_j\le m-1$), and
\begin{align*}
&u_j^{(s_j)} 
=(-)^m d^{-s_j-m/2} 
qu_j\,,  \quad
t^{(s_j)}
=\frac{\nu_j}{2}(l_j+1)^2+\frac{m-\nu_j}{2}l_j^2\,.
\end{align*}
\qed
\end{lem}
\medskip

We remark that 
Proposition \ref{repEm} holds true also for $m=1$ with the modification
\begin{align*}
&E_0(z)=a^{-1}
\sum_{j=0}^{n-1}u_j^{-1}:e^{A^{0,j}(z)}:
d^{\partial_{0,j}}
\,,
\quad F_0(z)=b^{-1}
\sum_{j=0}^{n-1}u_j:e^{B^{0,j}(z)}:
d^{-\partial_{0,j}}
\,,
\end{align*}
with $ab=q(1-q_1)(1-q_3)$.
In this paper we restrict ourselves to the case $m\ge2$. 
\medskip

\subsection{Representations of $\check{\E}_n$}\label{sec:dual-mn}
In a similar manner we can write a representation of $\check{\E}_n$ 
on the same Fock space. 
We write the generators of $\check{\E}_n$ with checks, as
$\Ec_i(z)$, $\Fc_i(z)$, $\Kc^\pm_i(z)$ and so forth. 

\begin{prop}\label{repEn}
Let $n\ge2$, $m\ge1$, 
$\uc_0,\ldots,\uc_{m-1}\in\C^\times$. 
The following formulas give a level $m$ action of 
$\check{\E}_n$ on $\mathbb{F}_{m,n}$:
\begin{align}
&q^{\check{\ve}_j}=q^{\ssec_j}
\,,
\quad \check{C}=q^m\,,\quad
\check{D}=q^{\deg}\,,
\label{dFock0}
\\
&\Hc_{j,r}=\sum_{i=0}^{m-1} q_3^{-ir}\bc^{i,j}_r\,,
\quad
\Hc_{j,-r}=\sum_{i=0}^{m-1}q^{-(m-1)r}q_1^{-ir}\bc^{i,j}_{-r}\,,
\label{dFock1}
\\
&\Ec_j(z)=\sum_{i=0}^{m-1} \uc_{i}^{-\delta_{j,0}} \Ec^{i,j}(z)\,,
\quad \Ec^{i,j}(z)=:e^{\Ac^{i,j}(z)}:\, \Uc^{i,j}(z)\,,
\label{dFock2}\\
&\Fc_j(z)=\sum_{i=0}^{m-1}\uc_{i}^{\delta_{j,0}}\Fc^{i,j}(z)\,,
\quad \Fc^{i,j}(z)=:e^{\Bc^{i,j}(z)}:\, \Vc^{i,j}(z)\,.
\end{align}
where $0\le i\le m-1$, $0\le j\le n-1$ and $r>0$.
We set for $1\le j\le n-1$
\begin{align}
&\Uc^{i,j}(z)=
e^{\e_{i,j}}e^{-\e_{i,j-1}}
(q^{m-1-i}z)^{-\partial_{i,j-1}+\partial_{i,j}+1}
\label{dFock4}\\
&\quad\times
\dc^{(1/2-j)\ssec_{j-1}+(1/2+j)\ssec_j-j(\partial_{i,j-1}-\partial_{i,j})}
q^{\sum_{s=i+1}^{m-1}(\partial_{s,j-1}-\partial_{s,j})}\,,
\nn\\
&\Vc^{i,j}(z)
=
e^{\e_{i,j-1}}e^{-\e_{i,j}}
(q^{i}z)^{\partial_{i,j-1}-\partial_{i,j}+1}
\label{dFock5}\\
&\quad\times
\dc^{-(1/2-j)\ssec_{j-1}-(1/2+j)\ssec_j+j(\partial_{i,j-1}-\partial_{i,j})}
q^{-\sum_{s=0}^{i-1}(\partial_{s,j-1}-\partial_{s,j})}
\,,\nn
\end{align}
and for $j=0$ 
\begin{align}
&\Uc^{i,0}(z)=
e^{\e_{i,0}}e^{-\e_{i,n-1}}
(q^{m-1-i}z)^{-\partial_{i,n-1}+\partial_{i,0}+1}
\label{dFock6}\\
&\quad\times
\dc^{(1/2-n)\ssec_{n-1}+(1/2)\ssec_0-n\partial_{i,n-1}}
q^{\sum_{s=i+1}^{m-1}(\partial_{s,n-1}-\partial_{s,0})
} 
\,,
\nn\\
&\Vc^{i,0}(z)
=
e^{\e_{i,n-1}}e^{-\e_{i,0}}
(q^{i}z)^{\partial_{i,n-1}-\partial_{i,0}+1}
\label{dFock7}\\
&\quad\times
\dc^{-(1/2-n)\ssec_{n-1}-(1/2)\ssec_0+n\partial_{i,n-1}}
q^{-\sum_{s=0}^{i-1}(\partial_{s,n-1}-\partial_{s,0})
}\,.
\nn
\end{align}
\qed
\end{prop}
\bigskip

\subsection{Quantum affine subalgebras}\label{sec:affine}
Recall that we have a vertical subalgebra $\U_m\subset \E_m$ 
isomorphic to the quantum affine algebra $U_q\widehat{\gl}_m$ generated by 
the currents $E_i(z),F_i(z)$, $1\le i\le m-1$, and $D$. We also have the vertical subalgebra 
$\check{\U}_n\subset \check{\E}_n$ 
generated by  $\Ec_i(z),\Fc_i(z)$ with $1\le i\le n-1$, 
and $\check{D}$. 

\begin{thm}\label{affine-com}
On the total Fock space $\mathbb{F}_{m,n}$, the action of the vertical 
subalgebras $\U_m$ and $\check{\U}_n$ mutually commute.  \qed
\end{thm}
Theorem \ref{affine-com} will be proved in Appendix \ref{sec:affine-com}. 
The picture is summarized in Fig. \ref{comm-action}. 

\begin{center}
\begin{figure}[H]
\begin{align*}
\xymatrix@R=6pt{
\E_m(q_1,q_2,q_3) &      &\E_n(q^\vee_1,q_2,q^\vee_3) \\
\bigcup         &       &  \bigcup\\
U_q\widehat{\gl}_m
\ar@/^14pt/[r]^-{\text{level $n$}}
& \mathbb{F}_{mn}& U_q\widehat{\gl}_n
\ar@/_14pt/[l]_-{\text{level $m$}}
} 
\end{align*}
\caption{Commuting actions of quantum affine algebras. The corresponding toroidal actions do
not commute.\label{comm-action}}
\end{figure}
\end{center}

The Fock space $\F^{(\nu)}_m(u;t,s)$ can be constructed as a semi-infinite wedge product of level zero vector representations, see \cite{FJMM3}.
The usual $(\gl_m,\gl_n)$ duality uses either symmetric or skew-symmetric powers of vector representations.
Therefore, it is natural to think that our construction is a quantum affine analog of the latter.

\section{$(\E_m,\check{\E}_n)$ duality}\label{sec:duality}

In this section, we present integrals of motion 
associated with quantum toroidal algebras
$\E_m$ and $\check{\E}_n$ following \cite{FJM}.  
We then state the duality. 

\subsection{Dressed currents}
Recall that we deal with algebras $\E_m=\E_m(q_1,q_2,q_3)$ and 
$\E^\vee_n=\E_n(\qc_1,q_2,\qc_3)$ where $q_1,q_2,\qc_1$ are independent 
parameters. We set
\begin{align}
&p=\qc_1^n\,,\quad 
p^*=\qc_3^{-n}\,, 
\qquad
\pc=q_1^m\,,  
\quad \pc^*=q_3^{-m}\,.   
\label{pppp}
\end{align}
We assume that $|p|,|p^*|,|\pc|,|\pc^*|<1$. 

Introduce the dressed currents
\begin{align}
&\mbE_i(z)=\mbK_i^{-}(z)^{-1} E_i(z) \,,
\quad
\mbK^-_i(z)=\prod_{l\ge0}\bar{K}^-_i(p^{*l}z)\,,
\label{dressedE}
\\
&\mbF_i(z)= F_i(z) \mbK_i^{+}(z)^{-1}\,,
\quad
\mbK^+_i(z)=\prod_{l\ge0}\bar{K}^+_i(p^{-l}z)\,,
\label{dressedF}
\end{align}
where $\bar{K}^\pm_i(z)=(K_i)^{\mp1}K^\pm_i(z)$, see \eqref{defK}. 
They are given respectively by formulas \eqref{Fock2} and \eqref{Fock3},  
replacing $A^{i,j}(z)$, $B^{i,j}(z)$ by  the currents
$\mbA^{i,j}(z)$, $\mbB^{i,j}(z)$ with Fourier coefficients
\begin{align}
&\mbA^{i,j}_r
=-\frac{1}{[r]}q^{-(n-1)r}\qc_3^{-jr}b^{i,j}_{r}\,,
\quad 
\mbA^{i,j}_{-r}=
\frac{1}{[r]}q^{(n-2)r}\qc_3^{jr}
\Bigl(b^{i,j}_{-r}
-\frac{1-q_2^r}{1-p^{*r}}\sum_{t=0}^{n-1}
\qc_3^{-tr}b^{i,j-t}_{-r}
\Bigr)
\,,
\label{Adr1}\\ 
&\mbB^{i,j}_r=
\frac{1}{[r]}q^r\qc_1^{jr}
\Bigl(b^{i,j}_r-\frac{1-q_2^{-r}}{1-p^r}
\sum_{t=0}^{n-1}\qc_1^{tr}b^{i,j+t}_r
\Bigr)
\,,
\quad 
\mbB^{i,j}_{-r}=-\frac{1}{[r]}\qc_1^{-jr}b^{i,j}_{-r}\,.
\label{Bdr1} 
\end{align}

The dual counterparts are
\begin{align}
&\check{\mbE}_i(z)=\check{\mbK}_i^{-}(z)^{-1} \check{E}_i(z) \,,
\quad
\check{\mbK}^-_i(z)=\prod_{l\ge0}\check{\bar{K}}^-_i(\pc^{*})^{l}z)\,,
\\
&\check{\mbF}_i(z)= \check{F}_i(z) \check{\mbK}_i^{+}(z)^{-1}\,,
\quad
\check{\mbK}^+_i(z)=\prod_{l\ge0}\check{\bar{K}}^+_i(\pc^{-l}z)\,,
\end{align}
with
\begin{align}
&
\check{\mbA}^{i,j}_r=
-\frac{1}{[r]}q^{-(m-1)r}q_3^{-ir}\bc^{i,j}_r\,,
\quad 
\check{\mbA}^{i,j}_{-r}=
\frac{1}{[r]}q^{(m-2)r}q_3^{ir}
\Bigl(\bc^{i,j}_{-r}
-\frac{1-q_2^r}{1-\pc^{*r}}
\sum_{s=0}^{m-1}q_3^{-sr}\bc^{i-s,j}_{-r}
\Bigr)\,,\label{Adr2}\\  
&\check{\mbB}^{i,j}_r=\frac{1}{[r]}
q^rq_1^{ir}
\Bigl(\bc^{i,j}_r-\frac{1-q_2^{-r}}{1-\pc^r}
\sum_{s=0}^{m-1}q_1^{sr}\bc^{i+s,j}_r\Bigr)\,,
\quad 
\check{\mbB}^{i,j}_{-r}=-\frac{1}{[r]}q_1^{-ir}\bc^{i,j}_{-r}\,.
\label{Bdr2} 
\end{align}

The currents $\mbA^{i,j}(z),\mbB^{i,j}(z)$ are $m$-periodic in $i$, whereas
 $\check{\mbA}^{i,j}(z),\check{\mbB}^{i,j}(z)$ are $n$-periodic in $j$.
They are quasi-periodic with respect to the other index,
\begin{align}
&\mbA^{i,j-n}(z)=\mbA^{i,j}(p^{*}z)\,,
\quad
\mbB^{i,j-n}(z)=\mbB^{i,j}(p z)\,,
\label{quasi1}\\
&\check{\mbA}^{i-m,j}(z)=\check{\mbA}^{i,j}(\pc^{*}z)\,,
\quad
\check{\mbB}^{i-m,j}(z)=\check{\mbB}^{i,j}(\pc z)\,.
\label{quasi2}
\end{align}
\medskip

\subsection{Integrals of motion}\label{subsec:IM}
Elliptic deformation of integrals of motion in CFT 
were constructed in \cite{FKSW}, \cite{FKSW1}, \cite{KS} 
as operators acting on Fock spaces. 
In \cite{FJM} they were identified with Taylor coefficients
of transfer matrices associated with quantum toroidal algebras. 
Let us briefly recall this.  

We introduce parameters $\bar p, \bar p_1,\ldots,\bar p_{m-1}$ and set 
\begin{align}
&p={\bar p}q^{-n},\quad p^*={\bar p}q^{n}\,,\quad 
p_i=\bar p_i q^{-\varepsilon_{i-1}+\varepsilon_i}  \,,
\quad 
p^*_i=\bar p_i q^{\varepsilon_{i-1}-\varepsilon_i}  \,,
\label{pi}
\\
&\bs p=(p,p_1,\ldots,p_{m-1})\,,\quad
\bs p^*=(p^*,p^*_1,\ldots,p^*_{m-1})\,, 
\label{bspp}
\end{align}
where $p,p^*$ are those in \eqref{pppp}.

Let $\cR$ be the universal $R$ matrix of $\E_m$ corresponding to the coproduct 
$\Delta$ in Appendix \ref{sec:Em}. Transfer matrices are weighted traces 
over $\F^{(\mu)}(u)=\F_m^{(\mu)}(u;0,0)$, 
\begin{align}
&{\bf T}_\mu(u;\bs p)=\Tr_{\F^{(\mu)}(u),1}
\Bigl(\bigl(\bar p^{{\bf d}}\prod_{i=1}^{n-1} {\bar p}_i^{\,-\bar\Lambda_i}\bigr)_1
\ \cR_{12}\Bigr)q^{{\bf d}}
\,,
\label{TM}\\
&{\bf T}^*_\mu(u;\bs p^*)=
\Tr_{\F^{(\mu)}(u),1}\Bigl(\bigl(
\bar p^{{\bf d}}\prod_{i=1}^{n-1} \bar p_i^{\, -\bar\Lambda_i} \bigr)_1
\ \cR_{21}^{-1}\Bigr)q^{-{\bf d}}\,.
\label{TM*}
\end{align}
Here we set $D=q^{\bf d}$, $\bar\Lambda_i=\varepsilon_1+\cdots+\varepsilon_i$, 
the suffixes $1,2$ refer to the tensor components and trace is taken on the first component. 

The transfer matrices \eqref{TM}, \eqref{TM*} are formal series in $u^{\mp1}$. 
Integrals of motion are defined as their coefficients (up to 
some normalization constants  $c_N,c_N^*$ which are  irrelevant here, see \cite{FJM} for the explicit formulas in cases of $m=1,2$).
\begin{align*}
&{\bf T}_\mu(u;\bs p)
=\sum_{M=0}^\infty u^{-M}c_M \mbG_{\mu,M}(\bs p)
\,,
\quad
{\bf T}^*_\mu(u;\bs p^*)
=\sum_{M=0}^\infty u^{M}c^*_M\cdot \mbG_{\mu,M}^*(\bs p^*)\,.
\end{align*}
We call $\mbG_{\mu,M}(\bs p)$, $\mbG^*_{\mu,M}(\bs p^*)$ 
integrals of motion of the first and the second kind, respectively. 
In order to write explicit formulas for them, we prepare some symbols.

Let $\vartheta^{(m)}_\mu$ denote the theta function associated with 
 the root lattice $\bar Q^{(m)}$ of $\mathfrak{sl}_m$: 
\begin{align*}
&\vartheta^{(m)}_\mu(z_1,\ldots,z_{m};\bs p)
=\sum_{\beta\in\bar Q^{(m)}+\bar\Lambda_\mu}
p^{(\beta,\beta)/2}\prod_{s=1}^{m-1}p_s^{-(\beta,\bar\Lambda_s)}
\prod_{i=1}^{m}z_i^{(\beta,\bar\alpha_i)}\,.
\end{align*}
Define functions $h_{\mu,M},h^*_{\mu,M}$ by 
\begin{align}
&h_{\mu,M}(x_{1,1},\ldots,x_{1,M},\ldots,x_{m,1},\ldots, x_{m,M};\bs p)
\label{hmu}\\
&=\frac{\prod_{i=1}^{m}\prod_{a<b}\Theta_p(x_{i,b}/x_{i,a},q_2x_{i,b}/x_{i,a})}
{\prod_{i=1}^{m-1}\prod_{a,b}\Theta_p(q_3^{-1}x_{i+1,b}/x_{i,a})
\cdot \prod_{a,b}\Theta_p(q_1x_{1,b}/x_{m,a})}
\nn\\
&\times\prod_{i=1}^{m}\prod_{a=1}^M x_{i,a}^{M-2a+1}
\cdot
\vartheta^{(m)}_\mu\Bigl(\prod_{a=1}^M x_{1,a},
\ldots,
\prod_{a=1}^M x_{m,a};\bs p\Bigr) 
\,,\nn\\
&h^*_{\mu,M}(x_{1,1},\ldots,x_{1,M},\ldots,x_{m,1},\ldots, x_{m,M};\bs p^*)
\label{hmu*}
\\
&=\frac{\prod_{i=1}^{m}\prod_{a<b}
\Theta_{p^*}(x_{i,b}/x_{i,a},q_2x_{i,b}/x_{i,a})}
{\prod_{i=1}^{m-1}\prod_{a,b}\Theta_{p^*}(q_3^{-1}x_{i+1,b}/x_{i,a})
\cdot \prod_{a,b}\Theta_{p^*}(q_1x_{1,b}/x_{m,a})}
\nn\\
&\times\prod_{i=1}^{m}\prod_{a=1}^M x_{i,a}^{M-2a+1}
\cdot
\vartheta^{(m)}_\mu\Bigl(\prod_{a=1}^M x_{1,a}^{-1},
\ldots,
\prod_{a=1}^M x_{m,a}^{-1};\bs p^*\Bigr)
\,.
\nn
\end{align}
We have the quasi-periodicity property
\begin{align*}
&\frac{h_{\mu,M}(x_{1,1},\ldots,p x_{i,a},\ldots,x_{m,M};\bs p)} 
{h_{\mu,M}(x_{1,1},\ldots,x_{i,a},\ldots,x_{m,M};\bs p)} 
=p_i q_2^{M-2a+1+M(\delta_{i,m}-\delta_{i,1})}\,,
\\
&\frac{h^*_{\mu,M}(x_{1,1},\ldots,p^* x_{i,a},\ldots,x_{m,M};\bs p^*)} 
{h^*_{\mu,M}(x_{1,1},\ldots,x_{i,a},\ldots,x_{m,M};\bs p^*)} 
=(p^*_i)^{-1} q_2^{M-2a+1+M(\delta_{i,m}-\delta_{i,1})}\,,
\end{align*}
where $p_m:=(p_1\cdots p_{m-1})^{-1}$, $p^*_m:=(p^*_1\cdots p^*_{m-1})^{-1}$.

\begin{prop}\label{prop:IM}\cite{FJM}
The integrals of motion $\mbG_{\mu,M}(\bs p)$, $\mbG^*_{\mu,M}(\bs p^*)$
are given by the following multiple integrals. 
\begin{align*}
\mbG_{\mu,M}(\bs p)&=
\int\!\!\cdots\!\!\int \prod_{i=1}^{m} \prod_{a=1}^M
\frac{dx_{i,a}}{2\pi \sqrt{-1}\,x_{i,a}}
\ 
\prod_{1\le a\le M}^{\curvearrowright}\mbF_1(x_{1,a})
\cdots 
\prod_{1\le a\le N}^{\curvearrowright}\mbF_{m}(x_{m,a})
\\
&\times h_{\mu,M}(x_{1,1},\ldots,x_{1,M},\ldots,x_{m,1},\ldots, x_{m,M};\bs p)
\,,\\
\mbG^*_{\mu,M}(\bs p^*)
&=
\int\!\!\cdots\!\!\int \prod_{i=1}^{m} \prod_{a=1}^M
\frac{dx_{i,a}}{2\pi\sqrt{-1}\,x_{i,a}}
\prod_{1\le a\le M}^{\curvearrowleft}\mbE_{m}(x_{m,a})
\cdots 
\prod_{1\le a\le M}^{\curvearrowleft}\mbE_1(x_{1,a})
\\
&\times h^{*}_{\mu,M}(x_{1,1},\ldots,x_{1,M},\ldots,x_{m,1},\ldots, x_{m,M};\bs p^*)
\,,
\end{align*}
where $\mbF_m(z)=\mbF_0(z)$, $\mbE_m(z)=\mbE_0(z)$.
The integrals in $\mbG_{\mu,M}(\bs p)$ (resp. $\mbG^*_{\mu,M}(\bs p^*)$)
are taken to be the unit circle $|x_{i,a}|=1$ when $|q_1|,|q_3|<1$ (resp. $|q_1|,|q_3|>1$), 
and by analytic continuation in the general case. 
\qed
\end{prop}
More details about the contours will be given in Subsection \ref{contours}.
\medskip

Operators 
$\mbG_{\mu,M}(\bs p)$, $\mbG^*_{\mu,M}(\bs p^*)$ 
act on each subspace of fixed `weight' $\bm$, 
\[
\C[\{a^{i,j}_{-r}\}_{r>0, 0\le i\le m-1\atop 0\le j\le n-1}] \otimes\ket{\bm}    
\subset \mathbb{F}_{m,n}\,.
\]
By construction they commute with each other,
\begin{align*}
[\mbG_{\mu,M}(\bs p),\mbG_{\nu,N}(\bs p)]=
[\mbG_{\mu,M}(\bs p),\mbG^*_{\nu,N}(\bs p^*)]=
[\mbG^*_{\mu,M}(\bs p^*),\mbG^*_{\nu,N}(\bs p^*)]=0\,.
\end{align*}
\medskip

\subsection{Duality}
In a similar manner, we have the dual counterparts of integrals of motion
\begin{align*}
&\check{\mbG}_{\nu,N}(\bs \pc)=
\int\!\!\cdots\!\!\int \prod_{j=1}^{n} \prod_{a=1}^N
\frac{dy_{j,a}}{2\pi\sqrt{-1}\,y_{j,a}}
\ 
\prod_{1\le a\le N}^{\curvearrowright}\check{\mbF}_1(y_{1,a})
\cdots 
\prod_{1\le a\le N}^{\curvearrowright}\check{\mbF}_{n}(y_{n,a})
\\
&\times \check{h}_{\nu,N}(y_{1,1},\ldots,y_{1,N},
\ldots,y_{n,1},\ldots, y_{n,N};\bs \pc)
\,,\\
&\check{\mbG}^*_{\nu,N}(\bs \pc^*)=
\int\!\!\cdots\!\!\int \prod_{j=1}^{n} \prod_{a=1}^N
\frac{dy_{j,a}}{2\pi\sqrt{-1}\,y_{j,a}}
\prod_{1\le a\le N}^{\curvearrowleft}\check{\mbE}_{n}(y_{n,a})
\cdots 
\prod_{1\le a\le N}^{\curvearrowleft}\check{\mbE}_1(y_{1,a})
\\
&\times \check{h}^{*}_{\nu,N}(y_{1,1},\ldots,y_{1,N},
\ldots,y_{n,1},\ldots, y_{n,N};\bs \pc^*)
\,.
\end{align*}
They depend on parameters ${\bs \pc}=(\pc,\pc_1,\ldots,\pc_{n-1})$,
 ${\bs \pc^*}=(\pc^*,\pc^*_1,\ldots,\pc^*_{n-1})$ related to 
$\bar \pc,\bar \pc_1,\ldots,\bar \pc_{n-1}$
in the same way as in \eqref{pi} (with $m$ and $n$ interchanged). 
\medskip

The following is the main result of this paper.
\begin{thm}\label{thm:IM}
Assume that the parameters $q,q_1,\qc_1,u_i,\uc_l$ are generic. 
Then we have
\begin{align*}
&[\mbG_{\mu,M}(\bs p),\check{\mbG}_{\nu,N}(\bs \pc)]=
[\mbG^*_{\mu,M}(\bs p^*),\check{\mbG}_{\nu,N}(\bs \pc)]
\\
&=
[\mbG_{\mu,M}(\bs p),\check{\mbG}^*_{\nu,N}(\bs \pc^*)]
=
[\mbG^*_{\mu,M}(\bs p^*),\check{\mbG}^*_{\nu,N}(\bs \pc^*)]=0\,
\end{align*}
for all $\mu\in\Z/m\Z,\nu\in\Z/n\Z$ and $M,N\ge1$, provided
\begin{align}
&\bar p_i=\frac{\uc_i}{\uc_{i-1}}
\, \quad (1\le i\le m-1)\,,
\\
& \bar\pc_l=\frac{u_l}{u_{l-1}}
\, \quad (1\le l\le n-1)\,.
\end{align}
\qed
\end{thm}

Proof of Theorem \ref{thm:IM} is done by direct computation similar to the one in  \cite{FKSW}, \cite{KS}. 
We collect the data necessary for the calculation  in Appendix \ref{sec:contraction}, 
and indicate the main steps of the proof in Appendix \ref{sec:cont}. 
\medskip

\begin{rem} {\rm
When one of the components of $\bs p$ is zero, namely, $\bar p=0$, see \eqref{pi}, \eqref{bspp}, 
the algebra of integrals of motion of $\E_k$ coincides 
with the Bethe algebra of the horizontal  $U_q\widehat{\gl}_k$ algebra.
The subspace of the top degree in the Fock module $\F_k^{(\nu)}(u;t,s)$ is preserved by the horizontal algebra and as $U_q\widehat{\gl}_k$ module is isomorphic to evaluation module associated to the $\nu$-th fundamental representation of $U_q{\gl}_k$. 
Thus Theorem \ref{thm:IM} implies the duality of $\gl_m$ and $\gl_n$  XXZ systems.  }
\end{rem}

\appendix

\section{Quantum toroidal algebra $\E_m$}\label{sec:Em}

In this section we give 
the definition of the algebra $\E_m=\E_m(q_1,q_2,q_3)$. 
Although we take $m\ge2$ in the main text, 
we include the definition for $m=1$ for completeness. 

Let $P$ be a free $\Z$ module with basis $\ve_i$, $i\in\Z/m\Z$, 
equipped with the inner product $(~,~):P\times P\to\Z$
such that $\ve_i$ are orthonormal. We set $\bar{\alpha}_i=\ve_{i-1}-\ve_{i}$.

Define further functions $g_{i,j}(z,w)$ by
\begin{align*}
m\ge 3&:\quad
g_{i,j}(z,w)=\begin{cases}
	      z-q_1w & (i\equiv j-1),\\
              z-q_2w & (i\equiv j),\\
	      z-q_3w & (i\equiv j+1),\\
              z-w & (i\not\equiv j,j\pm1),\\
	     \end{cases}\\
m=2&:\quad 
 g_{i,j}(z,w)=\begin{cases}
	      z-q_2w & (i\equiv j),\\
              (z-q_1w)(z-q_3w)& (i\not\equiv j),
	     \end{cases}\\
m=1&:\quad 
 g_{0,0}(z,w)=(z-q_1w)(z-q_2w)(z-q_3w),
\end{align*}
and set
$d_{i,j}=d^{\mp 1}$ ($i\equiv j\mp1, m\ge 3$), $=-1$ ($i\not\equiv j, m=2$), $=1$ (otherwise). 

By definition, $\E_m=\E_m(q_1,q_2,q_3)$ is a unital associative
algebra generated by $E_{i,k},F_{i,k},H_{i,r}$ 
and invertible elements $q^{h}$, $C$, $D$, 
where $i\in\Z/m\Z$, $k\in \Z$, $r\in\Z\backslash\{0\}$, $h\in P$.  
We write $K_i=q^{\bar\alpha_i}$. 
The defining relations are given in terms of generating series
\begin{align}
&E_i(z)=\sum_{k\in\Z}E_{i,k}z^{-k}, 
\quad
F_i(z)=\sum_{k\in\Z}F_{i,k}z^{-k},
\label{defEF}
\\
&K^{\pm}_i(z)=K_i^{\pm1}
\exp\Bigl(\pm(q-q^{-1})\sum_{r>0}H_{i,\pm r}z^{\mp r}\Bigr)\,. 
\label{defK}
\end{align}
The relations are as follows.
\bigskip

\noindent{\bf $C,K$ relations}
\begin{align*}
&\text{$C$ is central},\quad
q^{h}q^{h'}=q^{h+h'}\quad (h,h'\in P)\,,\quad
q^0=1\,,\quad
D q^h=q^h D\,,
\\
&q^{h}E_i(z)q^{-h}=
q^{(h,\bar\alpha_i)}E_i(z)\,,
\quad
q^{h}F_i(z)q^{-h}=
q^{-(h,\bar\alpha_i)}F_i(z)\,
\quad (h\in P)\,,
\\
&D E_i(z) D^{-1}=E_i(qz)\,,
\quad
D F_i(z) D^{-1}=F_i(qz)\,
\quad 
D K^{\pm}_i(z)D^{-1}=K^\pm_i(qz)\,,
\end{align*}
\bigskip

\noindent{\bf $H$-$E$ and $H$-$F$ relations}\quad
For $r\neq 0$, 
\begin{align*}
&[H_{i,r},E_j(z)]= a_{i,j}(r)C^{-(r+|r|)/2}  
\,z^r E_j(z)\,,
\\
&[H_{i,r},F_j(z)]=-a_{i,j}(r)C^{-(r-|r|)/2}   
\,z^r F_j(z)\,,
\\
&[H_{i,r},H_{j,s}]=\delta_{r+s,0} \cdot a_{i,j}(r)\,
\frac{C^r-C^{-r}}{q-q^{-1}}\,,
\end{align*}
where   
\begin{align*}
&a_{i,j}(r)
=\frac{[r]}{r}\
\Bigl(
(q^r+q^{-r})\delta^{(m)}_{i,j}-d^r\delta^{(m)}_{i,j-1}-d^{-r}\delta^{(m)}_{i,j+1}
\Bigr)\,.
\end{align*}
\bigskip

\noindent{\bf $E$-$F$ relations}
\begin{align*}
&[E_i(z),F_j(w)]=\frac{\delta_{i,j}}{q-q^{-1}}
(\delta\bigl(C\frac{w}{z}\bigr)K_i^+(w)
-\delta\bigl(C\frac{z}{w}\bigr)K_i^-(z))\,.
\end{align*}
\bigskip

\noindent{\bf $E$-$E$ and $F$-$F$ relations}
\begin{align*}
&[E_i(z),E_j(w)]=0\,, \quad [F_i(z),F_j(w)]=0\, \quad (i\not\equiv j,j\pm1)\,,
\\
&d_{i,j}g_{i,j}(z,w)E_i(z)E_j(w)+g_{j,i}(w,z)E_j(w)E_i(z)=0, 
\\
&d_{j,i}g_{j,i}(w,z)F_i(z)F_j(w)+g_{i,j}(z,w)F_j(w)F_i(z)=0.
\end{align*}
We omit the Serre relations which are not used in this paper.

\bigskip

The currents $E_i(z),F_i(z)$, $i\in\Z/m\Z$, 
together with $D$, generate 
a subalgebra $\U_m\subset \E_m$ 
isomorphic to the quantum affine algebra $U_q\widehat{\gl}_m$.
We call $\U_m$ the vertical subalgebra.

The elements $E_{i,0}, F_{i,0}, q^{h}$, $i\in\Z/m\Z$, $h\in P$, also generate 
a subalgebra isomorphic to the quantum affine algebra $U_q\widehat{\gl}_m$ which we call 
the horizontal subalgebra.

\bigskip
We use the following coproduct, which is {\it opposite} to the one used in \cite{FJM}:
\begin{align*}
&\Delta x=x\otimes x\quad (x=q^{\varepsilon_i}, C,D)\,,
\\
&\Delta E_i(z)=E_i(C_2z)\otimes K^{-}_i(z)+ 1\otimes E_i(z) \,,
\\
&\Delta F_i(z)=F_i(z)\otimes 1+K^{+}_i(z)\otimes F_i(C_1z) \,,
\\
&\Delta K^{+}_i(z)
=K^{+}_i(z)\otimes K^{+}_i(C_1z)\,,
\\
&\Delta K^{-}_i(z)
=K^{-}_i(C_2z)\otimes K^{-}_i(z)\,.
\end{align*}
Here $C_1=C\otimes 1$ and $C_2=1\otimes C$. 

Highest weight modules of $\E_m$ are defined in terms of generators 
$\theta^{-1}(E_i(z))$, $\theta^{-1}(F_i(z))$, 
$\theta^{-1}(K^\pm_i(z))$
obtained by applying Miki's automorphism $\theta$ which interchanges the vertical and horizontal 
subalgebras \cite{Mi}, see also \cite{FJMM}.
Let $\mathbf{P}= (P_0(z),\ldots,P_{m-1}(z))\in \C(z)^m$
be an $m$-tuple of rational functions which are regular at 
$z^{\pm1}=\infty$ and satisfy $P_i(0)P_i(\infty)=1$.
An $\E_m$ module is a highest weight module of highest weight $\mathbf{P}$
if it is generated by a vector $w$ satisfying
\begin{align*}
&\theta^{-1}(E_i(z))w=0\quad (i=0,1,\ldots,m-1)\,,\\ 
&\theta^{-1}(K^{\pm}_i(z))w=P_i(z)w\quad (i=0,1,\ldots,m-1)\,.
\end{align*}
In the last line, $P_i(z)$ stands for its expansion at $z^{\pm1}=\infty$. 

The following formulas are used to calculate highest weights of level one modules:
\begin{align*}
&\theta^{-1}\bigl(H_{i,1}\bigr)
=-(-d)^{-i}
[[\cdots[[\cdots[F_{0,0},F_{m-1,0}]_q\cdots,F_{i+1,0}]_q,F_{1,0}]_q\cdots F_{i-1,0}]_q,F_{i,0}]_{q^2},
\quad (1 \le i\le m-1)
\\
&\theta^{-1}\bigl(H_{0,1}\bigr)
=-(-d)^{-m+1}[[\cdots[F_{1,1},F_{2,0}]_q\cdots, F_{m-1,0}]_q,F_{0,-1}]_{q^2}\,,
\end{align*}
where $m\ge2$ and $[X,Y]_p=XY-pYX$.

\medskip

\section{Contractions}\label{sec:contraction}

In this section we present normal ordering rules 
for various currents. 
Each current is a product of an oscillator part and a zero mode part;
for instance, the oscillator part of $\mbE^{i,j}(z)$ 
is ${\mbE^{i,j}(z)}^{osc}=:e^{\mbA^{i,j}(z)}:$ and  the zero mode part
is $U^{i,j}(z)$. 
We say that a product of operators $X(z)Y(w)$ has contraction $f(z,w)$ 
if $X(z)Y(w)=f(z,w):X(z)Y(w):$.
We compute the contractions 
separately for the oscillator and the zero mode parts.

\subsection{Contractions for $\E_m$}\label{subsec:cont-Em}

\begin{lem}\label{cont-nondress}
Let $0\le i,k\le m-1$, $0\le j,l\le n-1$. The contractions of
the oscillator part of the non-dressed currents are given in 
Tables \ref{EEndr}--\ref{FEndr} below.
In all other cases the contraction is $1$.
\qed
\end{lem}

\noindent 

\begin{table}[H]  
\begin{center}
\caption{${E^{ij}(z)}^{osc}{E^{kl}(w)}^{osc}$\label{EEndr}}
\begin{tabular}{|c|c|c|c|}
\hline
 & $j<l$ & $j=l$ & $j>l$\\
\hline
$i\equiv k$ & $1$
& $(1-w/z)(1-q_2^{-1}w/z)$
& $(1-q_2^{-1}w/z)(1-q_2 w/z)^{-1}$
\\
\hline
$i+1\equiv k$ & $1$
& $(1-q_1w/z)^{-1}$
& $(1-q_3^{-1}w/z)(1-q_1w/z)^{-1}$
\\
\hline
$i-1\equiv k$ & $1$
& $(1-q_3w/z)^{-1}$
& $(1-q_1^{-1}w/z)(1-q_3w/z)^{-1}$
\\
\hline
\end{tabular} 
\end{center}
\end{table}
\bigskip

\begin{table}[H]
\begin{center}
\caption{${F^{ij}(z)}^{osc}{F^{kl}(w)}^{osc}$\label{FFndr}}
\begin{tabular}{|c|c|c|c|}
\hline
 & $j<l$ & $j=l$ & $j>l$\\
\hline
$i\equiv k$ & $1$
& $(1-w/z)(1-q_2w/z)$
& $(1-q_2w/z)(1-q_2^{-1}w/z)^{-1}$
\\
\hline
$i+1\equiv k$ & $1$
& $(1-q_3^{-1}w/z)^{-1}$
& $(1-q_1w/z)(1-q_3^{-1}w/z)^{-1}$
\\
\hline
$i-1\equiv k$ & $1$
& $(1-q_1^{-1}w/z)^{-1}$
& $(1-q_3 w/z)(1-q_3^{-1}w/z)^{-1}$
\\
\hline
\end{tabular} 
\end{center}
\end{table}
\bigskip

\begin{table}[H] 
\begin{center}
\caption{${E^{ij}(z)}^{osc}{F^{kl}(w)}^{osc}$\label{EFndr}}
\begin{tabular}{|c|c|c|}
\hline
 &$j=l$ & $j\neq l$
\\
\hline
$i\equiv k$ & $(1-q^{-n+2j}w/z)^{-1}(1-q^{-n+2j+2}w/z)^{-1}$
& 1
\\
\hline
$i+1\equiv k$ 
& $1-q^{-n+2j}q_3^{-1}w/z$
&1
\\
\hline
$i-1\equiv k$ 
&  $1-q^{-n+2j}q_1^{-1}w/z$
&1
\\
\hline
\end{tabular} 
\end{center}
\end{table}
\bigskip

\begin{table}[H] 
\begin{center}
\caption{${F^{kl}(w)}^{osc}{E^{ij}(z)}^{osc}$\label{FEndr}}
\begin{tabular}{|c|c|c|}
\hline
 &$j=l$ & $j\neq l$
\\
\hline
$i\equiv k$ & $(1-q^{n-2j}z/w)^{-1}(1-q^{n-2j-2}z/w)^{-1}$ & 1 \\
\hline
$i+1\equiv k$ & $1-q^{n-2j}q_3z/w$ &1 \\
\hline
$i-1\equiv k$ &  $1-q^{n-2j}q_1z/w$ &1 \\
\hline
\end{tabular} 
\end{center}
\end{table}
\bigskip


\begin{lem}\label{cont-auto-dress}
Let $0\le i,k\le m-1$, $0\le j,l\le n-1$. 
The contractions for the oscillator part of the dressed currents are given 
by Tables \ref{EE}--\ref{FE} below. 
In Table \ref{EE} we use $(z)_\infty=(z;p^*)_\infty$ while in 
Table \ref{FF} we use $(z)_\infty=(z;p)_\infty$. 
\qed
\end{lem}

\begin{table}[H] 
\begin{center}
\caption{${\mbE^{ij}(z)}^{osc}{\mbE^{kl}(w)}^{osc}$, 
$(z)_\infty=(z;p^*)_\infty$ \label{EE}}
\begin{tabular}{|c|c|c|c|}
\hline
 & $j<l$ & $j=l$ & $j>l$\\
\hline
$i\equiv k$ & $(q_2w/z)_\infty(q_2^{-1}w/z)_\infty^{-1}$
& $(1-w/z)(q_2w/z)_\infty (p^*q_2^{-1}w/z)_\infty^{-1}$
& $(p^*q_2w/z)_\infty (p^*q_2^{-1}w/z)_\infty^{-1}$
\\
\hline
$i+1\equiv k$ & $(q_1w/z)_\infty (q_3^{-1} w/z)_\infty^{-1}$
& $(p^*q_1w/z)_\infty (q_3^{-1}w/z)_\infty^{-1}$
& $(p^*q_1w/z)_\infty (p^*q_3^{-1}w/z)_\infty^{-1}$
\\
\hline
$i-1\equiv k$ & $(q_3w/z)_\infty (q_1^{-1} w/z)_\infty^{-1}$
& $(p^*q_3w/z)_\infty (q_1^{-1}w/z)_\infty^{-1}$
& $(p^*q_3w/z)_\infty (p^*q_1^{-1}w/z)_\infty^{-1}$
\\
\hline
\end{tabular} 
\end{center}
\end{table}
\bigskip

\begin{table}[H]  
\begin{center}
\caption{${\mbF^{ij}(z)}^{osc}{\mbF^{kl}(w)}^{osc}$, $(z)_\infty=(z;p)_\infty$
\label{FF}}
\begin{tabular}{|c|c|c|c|}
\hline
 & $j<l$ & $j=l$ & $j>l$\\
\hline
$i\equiv k$ & $(q_2^{-1}w/z)_\infty (q_2w/z)_\infty^{-1}$
& $(1-w/z)(q_2^{-1}w/z)_\infty (p q_2w/z)_\infty^{-1}$
& $(pq_2^{-1}w/z)_\infty (pq_2w/z)_\infty^{-1}$
\\
\hline
$i+1\equiv k$ & $(q_3^{-1}w/z)_\infty (q_1 w/z)_\infty^{-1}$
& $(pq_3^{-1}w/z)_\infty (q_1w/z)_\infty^{-1}$
& $(pq_3^{-1}w/z)_\infty (pq_1w/z)_\infty^{-1}$
\\
\hline
$i-1\equiv k$ & $(q_1^{-1}w/z)_\infty (q_3 w/z)_\infty^{-1}$
& $(p q_1^{-1}w/z)_\infty (q_3w/z)_\infty^{-1}$
& $(p q_1^{-1} w/z)_\infty (p q_3w/z)_\infty^{-1}$
\\
\hline
\end{tabular} 
\end{center}
\end{table}
\bigskip

\begin{table}[H]  
\begin{center}
\caption{${\mbE^{ij}(z)}^{osc}{\mbF^{kl}(w)}^{osc}$\label{EF}}
\begin{tabular}{|c|c|c|}
\hline
 &$j=l$ & $j\neq l$
\\
\hline
$i\equiv k$ & $(1-q^{-n+2j}w/z)^{-1}(1-q^{-n+2j+2}w/z)^{-1}$
& 1
\\
\hline
$i+1\equiv k$ 
& $1-q^{-n+2j}q_3^{-1}w/z$
&1
\\
\hline
$i-1\equiv k$ 
&  $1-q^{-n+2j}q_1^{-1}w/z$
&1
\\
\hline
\end{tabular} 
\end{center}
\end{table}
\bigskip

\begin{table}[H]  
\begin{center}
\caption{${\mbF^{kl}(w)}^{osc}{\mbE^{ij}(z)}^{osc}$\label{FE}}
\begin{tabular}{|c|c|c|}
\hline
 &$j=l$ & $j\neq l$
\\
\hline
$i\equiv k$ & $(1-q^{n-2j}z/w)^{-1}(1-q^{n-2j-2}z/w)^{-1}$
& 1
\\
\hline
$i+1\equiv k$ 
& $1-q^{n-2j}q_3 z/w$
&1
\\
\hline
$i-1\equiv k$ 
&  $1-q^{n-2j}q_1 z/w$
&1
\\
\hline
\end{tabular} 
\end{center}
\end{table}
\bigskip

\begin{lem}\label{cont-zerEm}
Set
\begin{align*}
&\bar a^{(m)}_{i,k}=-\delta^{(m)}_{i-1,k}+2\delta^{(m)}_{i,k}
-\delta^{(m)}_{i+1,k} \,,
\\
&p^{(m)}_{i,k}=
i\bar a^{(m)}_{i,k}+
m\delta^{(m)}_{i,0}(\delta^{(m)}_{k,0}-\delta^{(m)}_{k,-1})\,.
\end{align*}
The contractions of the zero mode part are given as follows. 
\begin{align}
&U^{i,j}(z) U^{k,l}(w)=:U^{i,j}(z) U^{k,l}(w):
\label{UUcont}\\ 
&\quad\times 
(q^{n-j-1}z)^{\bar a^{(m)}_{i,k}\delta^{(n)}_{j,l}}
d^{-(1/2)(\delta^{(m)}_{i-1,k}-\delta^{(m)}_{i+1,k})
-p^{(m)}_{i,k}(1-\delta^{(n)}_{j,l})}
q^{-\theta(j<l)\,\bar a^{(m)}_{i,k}}\,,
\nn\\
&V^{i,j}(z) V^{k,l}(w)=:V^{i,j}(z) V^{k,l}(w):
\label{VVcont}\\ 
&\quad\times 
(q^{j}z)^{\bar a^{(m)}_{i,k}\delta^{(n)}_{j,l}}
d^{-(1/2)(\delta^{(m)}_{i-1,k}-\delta^{(m)}_{i+1,k})
-p^{(m)}_{i,k}(1-\delta^{(n)}_{j,l})}
q^{-\theta(j>l)\,\bar a^{(m)}_{i,k}}
\,,
\nn\\
&U^{i,j}(z) V^{k,l}(w)=:U^{i,j}(z) V^{k,l}(w):
\label{UVcont}\\ 
&\quad\times 
(q^{n-j-1}z)^{-\bar a^{(m)}_{i,k}\delta^{(n)}_{j,l}}
d^{(1/2)(\delta^{(m)}_{i-1,k}-\delta^{(m)}_{i+1,k})
+p^{(m)}_{i,k}(1-\delta^{(n)}_{j,l})}
q^{\theta(j<l)\,\bar a^{(m)}_{i,k}}\,,
\nn\\
&V^{k,l}(w) U^{i,j}(z)=:V^{k,l}(w) U^{i,j}(z):
\label{VUcont}\\ 
&\quad\times 
(q^l w)^{-\bar a^{(m)}_{i,k}\delta^{(n)}_{j,l}}
d^{-(1/2)(\delta^{(m)}_{i-1,k}-\delta^{(m)}_{i+1,k})
+p^{(m)}_{k,i}(1-\delta^{(n)}_{j,l})}
q^{\theta(j<l)\,\bar a^{(m)}_{i,k}}\,.
\nn
\end{align}
\qed
\end{lem}

\subsection{Contractions between $\E_m$ and $\check{\E}_n$}\label{subsec:cont-EmEcn}

\begin{lem}\label{cont-dress}
Let $0\le i,k\le m-1$, $0\le j,l\le n-1$. 
The contractions of the 
oscillator part of the dressed currents 
with the dual ones are given by 
Tables \ref{EEc}--\ref{FcE} below.   
In all other cases the contractions are $1$. 
\qed
\end{lem}

\begin{table}[H]  
\begin{center}
\caption{${\mbE^{ij}(z)}^{osc}{\check{\mbE}^{kl}(w)}^{osc}$\label{EEc}}
\begin{tabular}{|c|c|c|}
\hline
 &$j\equiv l$ &  $j\equiv l-1$ \\
\hline
$i\equiv k$ & $(1-q^{m-n}\qc_3^{-j}q_3^k w/z)^{-1}$
&  $1-q^{m-n-2}\qc_3^{-j-1}q_3^k w/z$
\\
\hline
$i-1\equiv k$ 
& $1-q^{m-n+2}\qc_3^{-j}q_3^{k+1} w/z$
&  $(1-q^{m-n}\qc_3^{-j-1}q_3^{k+1}w/z)^{-1}$
\\
\hline
\end{tabular} 
\end{center}
\end{table}
\medskip

\begin{table}[H]  
\begin{center}
\caption{${\check{\mbE}^{kl}(w)}^{osc}{\mbE^{ij}(z)}^{osc}$\label{EcE}}
\begin{tabular}{|c|c|c|}
\hline
 &$j\equiv l$ &  $j\equiv l-1$ \\
\hline
$i\equiv k$ & $(1-q^{-m+n}\qc_3^{j}q_3^{-k} z/w)^{-1}$
&  $1-q^{-m+n+2}\qc_3^{j+1}q_3^{-k} z/w$
\\
\hline
$i-1\equiv k$ 
& $1-q^{-m+n-2}\qc_3^{j}q_3^{-k-1} z/w$
&  $(1-q^{-m+n}\qc_3^{j+1}q_3^{-k-1} z/w)^{-1}$
\\
\hline
\end{tabular} 
\end{center}
\end{table}

\bigskip

\begin{table}[H] 
\begin{center}
\caption{${\mbF^{ij}(z)}^{osc}\check{\mbF}^{kl}(w)^{osc}$\label{FFc}}
\begin{tabular}{|c|c|c|}
\hline
 &$j\equiv l$ &  $j\equiv l-1$ \\
\hline
$i\equiv k$ & $(1-\qc_1^{j}q_1^{-k} w/z)^{-1}$
&  $1-q^{2}\qc_1^{j+1}q_1^{-k} w/z$
\\
\hline
$i-1\equiv k$ 
& $1-q^{-2}\qc_1^{j}q_1^{-k-1} w/z$
&  $(1-\qc_1^{j+1}q_1^{-k-1}w/z)^{-1}$
\\
\hline
\end{tabular} 
\end{center}
\end{table}
\medskip

\begin{table}[H] 
\begin{center}
\caption{$\check{\mbF}^{kl}(w)^{osc}{\mbF^{ij}(z)}^{osc}$\label{FcF}}
\begin{tabular}{|c|c|c|}
\hline
 &$j\equiv l$ &  $j\equiv l-1$ \\
\hline
$i\equiv k$ & $(1-\qc_1^{-j}q_1^{k} z/w)^{-1}$
&  $1-q^{-2}\qc_1^{-j-1}q_1^{k} z/w$
\\
\hline
$i-1\equiv k$ 
& $1-q^{2}\qc_1^{-j}q_1^{k+1} z/w$
&  $(1-\qc_1^{-j-1}q_1^{k+1} z/w)^{-1}$
\\
\hline
\end{tabular} 
\end{center}
\end{table}
\bigskip

\begin{table}[H]  
\begin{center}
\caption{${\mbE^{ij}(z)}^{osc}{\check{\mbF}^{kl}(w)}^{osc}$\label{EFc}}
\begin{tabular}{|c|c|c|}
\hline
 &$j\equiv l$ &  $j\equiv l-1$ \\
\hline
$i\equiv k$ & $1-q^{-n+2}\qc_3^{-j}q_1^{-k} w/z$
&  $(1-q^{-n}\qc_3^{-j-1}q_1^{-k} w/z)^{-1}$
\\
\hline
$i-1\equiv k$ 
& $(1-q^{-n}\qc_3^{-j}q_1^{-k-1} w/z)^{-1}$
&  $1-q^{-n-2}\qc_3^{-j-1}q_1^{-k-1}w/z$
\\
\hline
\end{tabular} 
\end{center}
\end{table}
\bigskip

\begin{table}[H]  
\begin{center}
\caption{${\check{\mbF}^{kl}(w)}^{osc}{\mbE^{ij}(z)}^{osc}$\label{FcE}}
\begin{tabular}{|c|c|c|}
\hline
 &$j\equiv l$ &  $j\equiv l-1$ \\
\hline
$i\equiv k$ & $1-q^{n-2}\qc_3^{j}q_1^{k} z/w$
&  $(1-q^{n}\qc_3^{j+1}q_1^{k} z/w)^{-1}$
\\
\hline
$i-1\equiv k$ 
& $(1-q^{n}\qc_3^{j}q_1^{k+1} z/w)^{-1}$
&  $1-q^{n+2}\qc_3^{j+1}q_1^{k+1} z/w$
\\
\hline
\end{tabular} 
\end{center}
\end{table}
\bigskip

\begin{lem}\label{cont-zer}
Let $0\le i,k\le m-1$ and $0\le j,l\le n-1$, 
and set
\begin{align}
D^{(m)}_{i,k}=\delta^{(m)}_{i,k}-\delta^{(m)}_{i-1,k}\,,\quad 
D^{(n)}_{j,l}=\delta^{(n)}_{j,l}-\delta^{(n)}_{j,l-1}\,.
\label{Dmn}
\end{align}
The contractions of the zero mode part are given as follows. 
\begin{align*}
U^{i,j}(z) \Uc^{k,l}(w)&=:U^{i,j}(z) \Uc^{k,l}(w):\\
&
\times
(q^{n-1-j}z)^{-D^{(m)}_{i,k}D^{(n)}_{j,l}}
d^{(-k\delta^{(m)}_{i,k}+(k+1)\delta^{(m)}_{i-1,k})D^{(n)}_{j,l}}
q^{D^{(m)}_{i,k}(\delta^{(n)}_{j,l-1}-\delta^{(n)}_{0,l})
}\,,
\\
\Uc^{k,l}(w)U^{i,j}(z) &=: \Uc^{k,l}(w)U^{i,j}(z):\\
&
\times
(q^{m-1-k}w)^{-D^{(m)}_{i,k}D^{(n)}_{j,l}}
\dc^{\,D^{(m)}_{i,k}
(-j\delta^{(n)}_{j,l}+(j+1)\delta^{(n)}_{j,l-1})}
q^{(\delta^{(m)}_{i-1,k}-\delta^{(m)}_{i,0})D^{(n)}_{j,l}
}\,,
\\
V^{i,j}(z) \Vc^{k,l}(w)&=:V^{i,j}(z) \Vc^{k,l}(w):\\
&
\times
(q^{j}z)^{-D^{(m)}_{i,k}D^{(n)}_{j,l}}
d^{(-k\delta^{(m)}_{i,k}+(k+1)\delta^{(m)}_{i-1,k})D^{(n)}_{j,l}}
q^{-D^{(m)}_{i,k}(\delta^{(n)}_{j,l}-\delta^{(n)}_{0,l})
} \,,
\\
\Vc^{k,l}(w)V^{i,j}(z) &=: \Vc^{k,l}(w)V^{i,j}(z):\\
&
\times
(q^{k}w)^{-D^{(m)}_{i,k}D^{(n)}_{j,l}}
\dc^{\,D^{(m)}_{i,k}(-j\delta^{(n)}_{j,l}+(j+1)\delta^{(n)}_{j,l-1})}
q^{-(\delta^{(m)}_{i,k}-\delta^{(m)}_{i,0})D^{(n)}_{j,l}
}\,,
\\
U^{i,j}(z)\Vc^{k,l}(w) &=:U^{i,j}(z)\Vc^{k,l}(w):\\
&\times
(q^{n-1-j}z)^{D^{(m)}_{i,k}D^{(n)}_{j,l}}
d^{(k\delta^{(m)}_{i,k}-(k+1)\delta^{(m)}_{i-1,k})D^{(n)}_{j,l}}
q^{-D^{(m)}_{i,k}(\delta^{(n)}_{j,l-1}-\delta^{(n)}_{0,l})
}\,,
\\
\Vc^{k,l}(w)U^{i,j}(z) &=: \Vc^{k,l}(w)U^{i,j}(z):\\
&
\times
(q^{k}w)^{D^{(m)}_{i,k}D^{(n)}_{j,l}}
\dc^{\,D^{(m)}_{i,k}
(j\delta^{(n)}_{j,l}-(j+1)\delta^{(n)}_{j,l-1})}
q^{(\delta^{(m)}_{i,k}-\delta^{(m)}_{i,0})D^{(n)}_{j,l}
} \,.
\end{align*}
\qed
\end{lem}
\bigskip

The following Lemmas will be used in Section \ref{sec:cont}.
They can be proved by a straightforward calculation
using \eqref{bbc-rel2},\eqref{bbc-rel1} or
the definition \eqref{Fock4}--\eqref{Fock7} and \eqref{dFock4}--\eqref{dFock7}. 

\begin{lem}\label{ZerRel-ndr}
Assume that $1\le i\le m-1$ and $1\le l \le n-1$. 
Then we have
\begin{align}
&:E^{i,l}(z)\check{E}^{i,l}(w):+
q^{-2}:E^{i,l-1}(z)\check{E}^{i-1,l}(w):
=0
\quad 
 \text{if $w=q^{-m+n}q_3^{-i}\check{q}_3^{l}z$}\,,
\label{relUU}\\
&:F^{i,l}(z)\check{F}^{i,l}(w):+
q^{2}:F^{i,l-1}(z)\check{F}^{i-1,l}(w):
=0
\quad 
\text{if $w=q_1^{i}\check{q}_1^{-l}z$}\,,
\label{relVV}\\
&:E^{i,l}(z)\check{F}^{i-1,l}(w):+
q^{-2}:E^{i,l-1}(z)\check{F}^{i,l}(w):=0
\quad 
\text{if $w=q^{n}q_1^{i}\check{q}_3^{l}z$}\,.
\label{relUV}
\end{align}
\qed
\end{lem}
\bigskip

\begin{lem}\label{ZerRel}
Define
\begin{align*}
&\mbE^{i,-1}(z)=\mbE^{i,n-1}(p^{*}z)\,
\dc^{-n}q^{-\sse_{i-1}+\sse_i}\,,
\\
&\check{\mbE}^{-1,l}(z)=\check{\mbE}^{m-1,l}(\check{p}^{*}z)\,
d^{-m}q^{-\ssec_{l-1}+\ssec_l}\,,
\\
&\mbF^{i,-1}(z)=\mbF^{i,n-1}(p z)\,
\dc^{-n}q^{-\sse_{i-1}+\sse_i}\,,
\\
&\check{\mbF}^{-1,l}(z)=
\check{\mbF}^{m-1,l}(\pc z)\,
d^{-m}
q^{-\ssec_{l-1}+\ssec_l}\,.
\end{align*}
Then for $0\le i\le m-1$, $0\le l\le n-1$ we have 
\begin{align}
&:\mbE^{i,l}(z)\check{\mbE}^{i,l}(w):+
q^{-2}:\mbE^{i,l-1}(z)\check{\mbE}^{i-1,l}(w):
=0
\quad 
 \text{if $w=q^{-m+n}q_3^{-i}\check{q}_3^{l}z$}\,,
\label{EbEb}\\
&:\mbF^{i,l}(z)\check{\mbF}^{i,l}(w):+
q^{2}:\mbF^{i,l-1}(z)\check{\mbF}^{i-1,l}(w):
=0
\quad 
\text{if $w=q_1^{i}\check{q}_1^{-l}z$}\,,
\label{FbFb}\\
&:\mbE^{i,l}(z)\check{\mbF}^{i-1,l}(w):+
q^{-2}:\mbE^{i,l-1}(z)\check{\mbF}^{i,l}(w):=0
\quad 
\text{if $w=q^{n}q_1^{i}\check{q}_3^{l}z$}\,.
\label{EbFb}
\end{align}
\qed
\end{lem}
\medskip

\section{Proof of Theorem \ref{affine-com}}\label{sec:affine-com}
In this section we prove Theorem \ref{affine-com}. 
We begin with
\begin{lem}\label{HHc} 
Assume that $1\le i\le m-1$, $1\le l\le n-1$ and $r\neq 0$. Then 
\begin{align*}
[b^{i,t}_{r}, \check{H}_{t',-r}]=0\,,\quad [H_{s,r}, \bc^{s',l}_{-r}]=0
\quad (0\le s,s'\le m-1,  0\le t,t'\le n-1).
\end{align*}
\end{lem}
\begin{proof}
Suppose $r>0$. Using \eqref{bb3} we compute 
\begin{align*}
[H_{s,r},\bc^{s',l}_{-r}]=-\frac{[r]^2}{r}q^{r}
\bigl(q_3^r\delta^{(m)}_{s-1,s'}-\delta^{(m)}_{s,s'}\bigr)
\sum_{j=0}^{n-1}
\bigl(
\qc_1^{(j+1)r}
\delta^{(n)}_{j,l-1}-
\qc_1^{jr}
\delta^{(n)}_{j,l}\bigr)=0\,,
\end{align*}
where we use $1\le l\le n-1$. The other cases are similar.
\end{proof}
\medskip

\begin{lem}\label{cont-dress-nondress}
Assume that 
$1\le i,k\le m-1$, $1\le j,l\le n-1$. 
Then formulas for the contractions in Tables \ref{EE}--\ref{FE} 
hold true if we replace
${\mbE^{i,j}(z)}^{osc}$, ${\mbF^{i,j}(z)}^{osc}$, 
${\check{\mbE}^{k,l}(w)}^{osc}$, ${\check{\mbF}^{k,l}(w)}^{osc}$
by
${E^{i,j}(z)}^{osc}$, ${F^{i,j}(z)}^{osc}$, 
${\check{E}^{k,l}(w)}^{osc}$, ${\check{F}^{k,l}(w)}^{osc}$, 
respectively.
\end{lem}
\begin{proof}
By the definition we have
\begin{align*}
&A_r^{i,j}=\mbA^{i,j}_r\,,\quad
A^{i,j}_{-r}=\mbA^{i,j}_{-r}-\frac{q-q^{-1}}{1-p^{*\,r}}H_{i,-r}\,,
\\
&B^{i,j}_r=\mbB^{i,j}_r+\frac{q-q^{-1}}{1-p^{r}}H_{i,r}\,,
\quad
B^{i,j}_{-r}=\mbB^{i,j}_{-r}\,,
\end{align*}
and similarly for those with checks. 
Hence the assertion follows from 
Lemma \ref{HHc}. 
\end{proof}
\medskip

From Lemma \ref{HHc} we obtain the commutativity
\begin{align*}
&[H_{i,r},\check{H}_{l,r'}]=0\,,\\
&[H_{i,r},\check{E}_{l}(w)]=[H_{i,r},\check{F}_{l}(w)]=0\,,\\
&[E_i(z),\check{H}_{l,r}]=[F_i(z),\check{H}_{l,r}]=0\,
\end{align*}
provided $i, l\neq 0$. 
To show
Theorem \ref{affine-com} it remains to check the following. 
\begin{lem}\label{EFEcFc}
Assuming  $1\le i\le m-1$, $1\le l\le n-1$ 
we have 
\begin{align*}
&[E_i(z),\Ec_l(w)]=0\,, 
\quad [F_i(z),\Fc_l(w)]=0\,, \\
&[E_i(z),\Fc_l(w)]=0\,,\quad [\Ec_i(z),F_l(w)]=0\,.  
\end{align*} 
\end{lem}
\begin{proof}
Lemma \ref{cont-nondress} and Lemma \ref{cont-zer} 
allow us to compute 
the commutators of $E^{i,j}(z),F^{i,j}(z)$ with 
$\Ec^{k,l}(w),\Fc^{k,l}(w)$.  
In each case only two terms survive, with the result
\begin{align}
&[E^{i,j}(z),\Ec^{k,l}(w)]=
\delta\Bigl(q^{m-n}q_3^{i}\qc_3^{-l}\frac{w}{z}\Bigr)
(q^{m-1-i}\dc^{\,l} w)^{-1}
\label{cancelEE}\\
&\quad \times \Bigl(
\delta_{i,k}\delta_{j,l}:E^{i,l}(z)\Ec^{i,l}(w):
+
\delta_{i-1,k}\delta_{j,l-1} q^{-2}:E^{i,l-1}(z)\Ec^{i-1,l}(w):
\Bigr)\,,
\nn\\
&[F^{i,j}(z),\Fc^{k,l}(w)]=
\delta\Bigl(q_1^{-i}\qc_1^{l}\frac{w}{z}\Bigr)(q^{i-1}\dc^{\,l} w)^{-1}
\label{cancelFF}\\
&\quad\times
\Bigl(q^{-2}\delta_{i,k}\delta_{j,l}:F^{i,l}(z)\Fc^{i,l}(w):
+
\delta_{i-1,k}\delta_{j,l-1} :F^{i,l-1}(z)\Fc^{i-1,l}(w):
\Bigr)\,,
\nn\\
&[E^{i,j}(z),\Fc^{k,l}(w)]=
\delta\Bigl(q^{-n}q_1^{-i}\qc_3^{-l}\frac{w}{z}\Bigr)
(q^{i-1}\dc^{\,l}w)^{-1}
\label{cancelEF}\\
&\quad\times
\Bigl(
\delta_{i-1,k}\delta_{j,l}:E^{i,l}(z)\Fc^{i-1,l}(w):
+
q^{-2}\delta_{i,k}\delta_{j,l-1}:E^{i,l-1}(z)\Fc^{i,l}(w):
\Bigr)\,.
\nn
\end{align}
Using Lemma \ref{ZerRel-ndr}
and summing over $j,k$, we obtain 
the desired equalities.
\end{proof}
\medskip

\section{Commutativity of $\mbG_{\mu,M}(\bs p)$ and 
$\check{\mbG}_{\nu,N}(\bs \pc)$
}\label{sec:cont}
In this section we prove Theorem \ref{thm:IM}.  
Since the working is all very similar, 
we illustrate the proof for the simplest integrals of motion of the first kind, namely 
\begin{align*}
&\mbG_{\mu,1}(\bs p) =
\int\!\!\cdots\!\!\int_C \prod_{s=1}^{m} 
\frac{dx_{s}}{2\pi\sqrt{-1}\,x_{s}}\,
\mbF_1(x_{1})\cdots \mbF_{m}(x_{m})\,
h_{\mu,1}(x_{1},\ldots,x_{m};\bs p)\,,
\\
&\check{\mbG}_{\nu,1}(\bs \pc) =
\int\!\!\cdots\!\!\int_{\check{C}} \prod_{t=1}^{n} 
\frac{dy_{t}}{2\pi\sqrt{-1}\,y_{t}}\,
\check{\mbF}_1(y_{1})\ldots \check{\mbF}_{n}(y_{n})\,
\check{h}_{\nu,1}(y_1,\ldots,y_n;\bs \pc)\,.
\end{align*}

\subsection{Contours}\label{contours}
The integration cycles $C,\check{C}$ are specified as follows. 

For $\mbG_{\mu,1}(\bs p)$, 
the poles of the integrand come from the denominators of 
$h_{\mu,1}(x_{1},\ldots,x_{m};\bs p)$ 
given in \eqref{hmu}, 
as well as from the contractions given in Table \ref{FF}.
Altogether the integrand takes the form
\begin{align*}
&\mbF_1(x_{1})\cdots \mbF_{m}(x_{m})\,
h_{\mu,1}(x_{1},\ldots,x_{m};\bs p)
 \\
&=\sum_{j_1,\ldots,j_m}
:\mbF^{1,j_1}(x_{1})\cdots \mbF^{m,j_m}(x_{m}):
\frac{\varphi_{j_1,\ldots,j_m}(x_1,\ldots,x_m)}
{\prod_{j=1}^m(q_1x_{j+1}/x_j,q_3x_j/x_{j+1};p)_\infty}
\end{align*}
with some holomorphic function $\varphi_{j_1,\ldots,j_m}(x_1,\ldots,x_m)$.
Here we set $x_0=x_m$ and $x_{m+1}=x_1$. 
For each $i$, the poles with respect to the variable $x_i$ consist of
two groups, 
\begin{align*}
&p^kq_3x_{i-1}\,,\quad p^kq_1x_{i+1}\quad (k\ge0)\,,\\
&p^{-k}q_1^{-1}x_{i-1}\,,\quad p^{-k}q_3^{-1}x_{i+1}\,, \quad (k\ge0)\,.
\end{align*}
The cycle $C$ is such that $x_i$ 
encircles the first group separating it from the second, see Figure \ref{Fig0}
below:

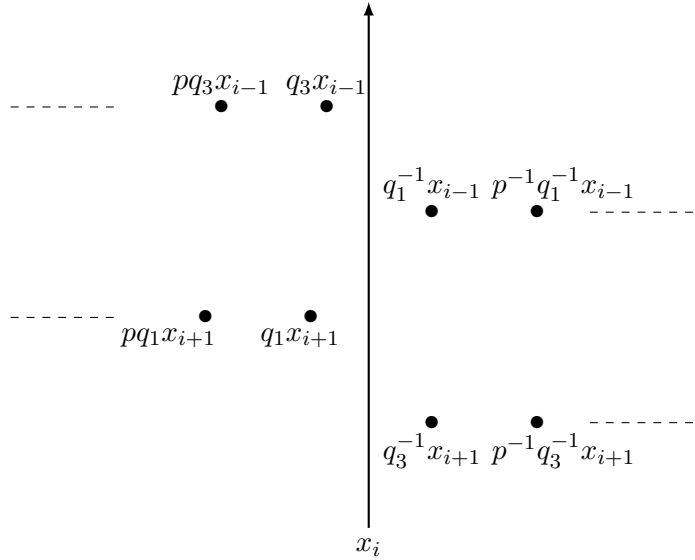
\begin{figure}[H]
\begin{center}
\begin{tikzpicture}[scale=0.7]

\draw[dashed] (-9,5)--(-7,5);
\node at (-5,5){$\bullet$};
\node at (-3,5){$\bullet$};
\node  [above] at (-5,5){$pq_3x_{i-1}$};
\node [above] at (-3,5){$q_3x_{i-1}$}; 

\node at (-1,3){$\bullet$};
\node at (1,3){$\bullet$};
\draw[dashed] (2,3)--(4,3);
\node  [above] at  (-1,3) {\small $q_1^{-1}x_{i-1}$};
\node  [above] at  (1.5,3) {\small $p^{-1}q_1^{-1}x_{i-1}$};

\draw[dashed] (-9,1)--(-7,1);
\node at (-5.3,1){$\bullet$};
\node at (-3.3,1){$\bullet$};
\node  [below] at  (-6,1) {\small $pq_1x_{i+1}$};
\node  [below] at  (-3.5,1) {\small $q_1x_{i+1}$};

\node at (-1,-1){$\bullet$};
\node at (1,-1){$\bullet$};

\draw[dashed] (2,-1)--(4,-1);
\node  [below] at  (-1,-1) {\small $q_3^{-1}x_{i+1}$};
\node  [below] at  (1.5,-1) {\small $p^{-1}q_3^{-1}x_{i+1}$};

\node [below] at (-2.2,-3) {\small $x_i$};
\draw [-latex,thick] (-2.2,-3)--(-2.2,7);
\end{tikzpicture}
\caption{The contour in the $x_i$ plane. \label{Fig0}}
\end{center}
\end{figure}

Similarly, 
$\check{C}$ is such that $y_l$
encircles $\pc^k\qc_3y_{l-1}$, $\pc^k\qc_1y_{l+1}$ keeping 
$\pc^{-k}\qc_1^{-1}y_{l-1}$, $\pc^{-k}\qc_3^{-1}y_{l+1}$ outside
for all $k\ge0$, where $y_0=y_n$ and $y_{n+1}=y_1$.    

\subsection{Commutators}\label{FFFF}
First we note the following relations of formal series, which follow from 
Lemma \ref{cont-dress} and Lemma \ref{cont-zer}.

\begin{lem}\label{EFcom}
(1) If $i\equiv k \bmod m$ and $j\equiv l\bmod n$, then 
\begin{align*}
&\mbE^{i,j}(z)\check{\mbE}^{k,l}(w)
-q^{\delta^{(m)}_{i,0}-\delta^{(n)}_{0,l}}\check{\mbE}^{k,l}(w)\mbE^{i,j}(z)
=:\mbE^{i,j}(z)\check{\mbE}^{k,l}(w):
\\
&
\quad\times
\delta\Bigl(q^{m-n}q_3^k\qc_3^{-j}\frac{w}{z}
\Bigr)\bigl(q^{n-1-j}z\bigr)^{-1}d^{-k}q^{-\delta^{(n)}_{0,l}}\,,
\\
&\mbF^{i,j}(z)\check{\mbF}^{k,l}(w)
-q^{-\delta^{(m)}_{i,0}+\delta^{(n)}_{0,l}}\check{\mbF}^{k,l}(w)\mbF^{i,j}(z)
=:\mbF^{i,j}(z)\check{\mbF}^{k,l}(w):
\\
&
\quad\times
\delta\Bigl(q_1^{-k}\qc_1^{j}\frac{w}{z}
\Bigr)\bigl(q^{j}z\bigr)^{-1}d^{-k}q^{-1+\delta^{(n)}_{0,l}}
\,.
\end{align*}

(2) If $i-1\equiv k \bmod m$ and $j\equiv l-1\bmod n$, then 
\begin{align*}
&\mbE^{i,j}(z)\check{\mbE}^{k,l}(w)
-q^{-\delta^{(m)}_{i,0}+\delta^{(n)}_{0,l}}\check{\mbE}^{k,l}(w)\mbE^{i,j}(z)
=:\mbE^{i,j}(z)\check{\mbE}^{k,l}(w):
\\
&
\quad\times
\delta\Bigl(q^{m-n}q_3^{k+1}\qc_3^{-j-1}\frac{w}{z}
\Bigr)
\bigl(q^{n-1-j}z\bigr)^{-1}d^{-k-1}q^{-1+\delta^{(n)}_{0,l}}\,,
\\
&\mbF^{i,j}(z)\check{\mbF}^{k,l}(w)
-q^{\delta^{(m)}_{i,0}-\delta^{(n)}_{0,l}}\check{\mbF}^{k,l}(w)\mbF^{i,j}(z)
=:\mbF^{i,j}(z)\check{\mbF}^{k,l}(w):
\\
&
\quad\times
\delta\Bigl(q_1^{-k-1}\qc_1^{j+1}\frac{w}{z}
\Bigr)
\bigl(q^{j}z\bigr)^{-1}d^{-k-1}q^{-\delta^{(n)}_{0,l}}
\,.
\end{align*}

(3) If  $i\equiv k \bmod m$ and $j\equiv l-1\bmod n$, then 
\begin{align*}
&\mbE^{i,j}(z)\check{\mbF}^{k,l}(w)
-q^{-\delta^{(m)}_{i,0}+\delta^{(n)}_{0,l}}\check{\mbF}^{k,l}(w)\mbE^{i,j}(z)
=:\mbE^{i,j}(z)\check{\mbF}^{k,l}(w):
\\
&
\quad\times
\delta\Bigl(q^nq_1^{k}\qc_3^{j+1}\frac{z}{w}
\Bigr)\bigl(q^{n-1-j}z\bigr)^{-1}d^{-k}q^{-1+\delta^{(n)}_{0,l}}
\,,
\end{align*}
and if  $i-1\equiv k \bmod m$ and $j\equiv l\bmod n$, then 
\begin{align*}
&\mbE^{i,j}(z)\check{\mbF}^{k,l}(w)
-q^{\delta^{(m)}_{i,0}-\delta^{(n)}_{0,l}}\check{\mbF}^{k,l}(w)\mbE^{i,j}(z)
=:\mbE^{i,j}(z)\check{\mbF}^{k,l}(w):
\\
&
\quad\times
\delta\Bigl(q^nq_1^{k+1}\qc_3^{j}\frac{z}{w}
\Bigr)
\bigl(q^{n-1-j}z\bigr)^{-1}d^{-k-1}q^{-\delta^{(n)}_{0,l}}
\,.
\end{align*}

(4) In all cases other than those given above, we have
\begin{align*}
&\mbE^{i,j}(z)\check{\mbE}^{k,l}(w)=\check{\mbE}^{k,l}(w)\mbE^{i,j}(z)
\times q^{-D^{(m)}_{i,k}\delta^{(n)}_{0,l}+D^{(n)}_{j,l}\delta^{(m)}_{i,0}}\,,
\\
&\mbF^{i,j}(z)\check{\mbF}^{k,l}(w)=\check{\mbF}^{k,l}(w)\mbF^{i,j}(z)
\times q^{D^{(m)}_{i,k}\delta^{(n)}_{0,l}-D^{(n)}_{j,l}\delta^{(m)}_{i,0}}\,,
\\
&\mbE^{i,j}(z)\check{\mbF}^{k,l}(w)=\check{\mbF}^{k,l}(w)\mbE^{i,j}(z)
\times q^{D^{(m)}_{i,k}\delta^{(n)}_{0,l}+D^{(n)}_{j,l}\delta^{(m)}_{i,0}}\,.
\end{align*}
\end{lem}
We recall that $D^{(m)}_{i,k}, D^{(n)}_{j,l}$ are given in \eqref{Dmn}.
\medskip

Our task is to check the vanishing of the commutator
\begin{align*}
&[\mbG_{\mu,1}(\bs p),\check{\mbG}_{\nu,1}(\bs \pc)]
=
\int\!\!\cdots\!\!\int_C \prod_{s=1}^{m} 
\frac{dx_{s}}{2\pi\sqrt{-1} x_{s}}\,
\int\!\!\cdots\!\!\int_{\check{C}} \prod_{t=1}^{n} 
\frac{dy_{t}}{2\pi\sqrt{-1}y_{t}}\,
\\
&\times
[\mbF_1(x_{1})\cdots \mbF_{m}(x_{m}),
\check{\mbF}_1(y_{1})\ldots \check{\mbF}_{n}(y_{n})]
h_{\mu,1}(x_{1},\ldots,x_{m};\bs p)
\check{h}_{\nu,1}(y_1,\ldots,y_n;\bs \pc)\,
\end{align*}
in the sense of matrix elements. 
\medskip

We begin by rewriting 
\begin{align*}
&[\mbF_1(x_{1})\cdots \mbF_{m}(x_{m}),
\check{\mbF}_1(y_{1})\ldots \check{\mbF}_{n}(y_{n})]
\\
&=\sum_{0\le j_1,\ldots,j_m\le n-1}\sum_{0\le k_1,\ldots,k_n\le m-1}
u_{j_m}\check{u}_{k_n}
[\mbF^{1,j_1}_1(x_{1})\cdots \mbF^{m,j_m}_{m}(x_{m}),
\check{\mbF}^{k_1,1}_1(y_{1})\ldots \check{\mbF}^{k_n,n}_{n}(y_{n})]\,.
\end{align*}
For $\bj=(j_1,\ldots,j_m)$, $\bk=(k_1,\ldots,k_n)$, let 
$c_{a,b}(\bj,\bk)=D^{(m)}_{a,k_b}\delta^{(n)}_{0,b}-D^{(n)}_{j_a,b}\delta^{(m)}_{a,0}$. 
Pushing $\mbF^{m,j_m}_{m}(x_{m})$ through to the right, we obtain
\begin{align*}
&\mbF^{m,j_m}(x_{m})\cdot 
\prod_{1\le t\le n}^{\curvearrowright}\check{\mbF}^{k_t,t}(y_t)
\\
&=\sum_{l=1}^n q^{\sum_{b=1}^{l-1}c_{m,b}(\bj,\bk)}
\prod_{1\le t\le l-1}^{\curvearrowright}\check{\mbF}^{k_t,t}(y_t)
\cdot
[\mbF^{m,j_m}(x_{m}), \check{\mbF}^{k_{l},l}(y_{l})]_{q^{c_{m,l}(\bj,\bk)}}
\cdot
\prod_{l+1\le t\le n}^{\curvearrowright}\check{\mbF}^{k_t,t}(y_t)
\\
&+ q^{D^{(m)}_{m,k_n}}
\prod_{1\le t\le n}^{\curvearrowright}\check{\mbF}^{k_t,t}(y_t)
\cdot\mbF^{m,j_m}(x_{m})\,,
\end{align*}
where we use $\sum_{b=1}^nc_{a,b}(\bj,\bk)=D^{(m)}_{a,k_n}$.  
Continuing the same way and noting that $\sum_{a=1}^mD^{(m)}_{a,k_n}=0$, 
we arrive at 
\begin{align*}
&[\mbF_1(x_{1})\cdots \mbF_{m}(x_{m}),
\check{\mbF}_1(y_{1})\ldots \check{\mbF}_{n}(y_{n})]
\\
&=\sum_{0\le j_1,\ldots,j_m\le n-1}\sum_{0\le k_1,\ldots,k_n\le m-1}
u_{j_m}\check{u}_{k_n}
\sum_{i=1}^m\sum_{l=1}^n 
 q^{\sum_{b=1}^{l-1}c_{i,b}(\bj,\bk)+\sum_{a=i+1}^mD^{(m)}_{a,k_n}}\\
&\times
\prod_{1\le s\le i-1}^{\curvearrowright}{\mbF}^{s,j_s}(x_{s})
\prod_{1\le t\le l-1}^{\curvearrowright}\check{\mbF}^{k_t,t}(y_t)
\cdot
[\mbF^{i,j_i}(x_{i}), \check{\mbF}^{k_{l},l}(y_{l})]_{q^{c_{i,l}(\bj,\bk)}}
\cdot 
\prod_{l+1\le t\le n}^{\curvearrowright}\check{\mbF}^{k_t,t}(y_t)
\prod_{i+1\le s\le m}^{\curvearrowright}{\mbF}^{s,j_s}(x_{s})\,.
\end{align*}
By Lemma \ref{EFcom}, the factor
$[\mbF^{i,j_i}(x_{i}), \check{\mbF}^{k_{l},l}(y_{l})]_{q^{c_{i,l}(\bj,\bk)}}$
is a sum of two terms containing delta functions. 
By the same argument as in the proof of Lemma \ref{EFEcFc}, 
these terms cancel out under the integral if $1\le i\le m-1$ and $1\le l\le n-1$.
It remains to consider the three cases
\begin{enumerate}
 \item $1\le i\le m-1$, $l=n$
\item $i=m$, $1\le l\le n-1$
\item $i=m$, $l=n$
\end{enumerate}

\medskip

\subsection{Commutativity of $\mbG_{\mu,1}$ and $\check{\mbG}_{\nu,1}$}
We now consider the case (i) in some detail. 
By Lemma \ref{EFEcFc}, non-trivial contributions arise only from terms with 
$(j_i,k_l)=(0,i), (n-1,i-1)$. Apart from a common factor, they give a sum $I+II$ with
\begin{align*}
I=&\check{u}_i\delta\bigl(q_1^{-i}y_n/x_i\bigr) x_i^{-1}
:\mbF^{i,0}(x_i)\check{\mbF}^{i,0}(q_1^ix_i):q^{-1}
\,,\\
II=&\check{u}_{i-1}\delta\bigl(q_1^{-i}py_n/x_i\bigr) q^{-n}x_i^{-1}
:\mbF^{i,n-1}(x_i)\check{\mbF}^{i-1,0}(p^{-1}q_1^ix_i):q^{\delta_{i,1}}\,.
\end{align*}
The powers $q^{-1}$, $q^{\delta_{i,1}}$ are due to 
$q^{\sum_{a=i+1}^mD^{(m)}_{a,k_n}}$. 

Let us examine the pole structure of the term $I$. 
With respect to $y_n$, 
the product $\mbF^{i,0}(x_i)\check{\mbF}^{i,0}(y_n)$ 
has a pole at $q_1^ix_i$ which corresponds to the delta function. In addition,
there is another pole at $y_n=p^{-1}q_1^{i+1}x_{i+1}$ coming from 
the contraction of $\check{\mbF}^{i,0}(y_n)\mbF^{i+1,n-1}(x_{i+1})$. 
Upon taking the residue at $y_n=q_1^ix_i$, a new pole 
$p^{-1}q_ix_{i+1}$ is produced with respect to $x_i$.  
This pole must be inside the contour of the $x_i$ integral
because analytic continuation is made from the region
$|y_n|\gg |x_{i+1}|$. The pole $x_i=q_3^{-1}x_{i+1}$ 
comes only from  the contraction of $\mbF^{i,0}(x_i)\mbF^{i+1,0}(x_{i+1})$. 
On the other hand, the contraction of 
$\check{\mbF}^{i,0}(y_n)\mbF^{i+1,0}(x_{i+1})$ gives 
$1-q^2q_1^{i+1}x_{i+1}/y_n$, which cancels the pole $x_i=q_3^{-1}x_{i+1}$  
after the residue is taken in $y_n$. 
The situation is summarized in Figure \ref{Fig1}. 

\begin{figure}[H]
\begin{center}
\begin{tikzpicture}[scale=0.7]
\draw[dashed] (-9,5)--(-7,5);
\node at (-5,5){$\bullet$};
\node at (-3,5){$\bullet$};
\node  [above] at (-5,5){$pq_3x_{i-1}$};
\node [above] at (-3,5){$q_3x_{i-1}$}; 
\node at (-1,3){$\bullet$};
\node at (1,3){$\bullet$};
\draw[dashed] (2,3)--(4,3);
\node  [above] at  (-1,3) {\small $q_1^{-1}x_{i-1}$};
\node  [above] at  (1.5,3) {\small $p^{-1}q_1^{-1}x_{i-1}$};
\draw[dashed] (-9,1)--(-7,1);
\node at (-5.3,1){$\bullet$};
\node at (-3.3,1){$\bullet$};
\node at (-1.3,1){\large $\star$};
\node  [below] at  (-6,1) {\small $pq_1x_{i+1}$};
\node  [below] at  (-3.5,1) {\small $q_1x_{i+1}$};
\node  [below] at  (-1.3,1) {\small $p^{-1}q_1x_{i+1}$};
\node at (-1,-1){$\circ$};
\node at (1,-1){$\bullet$};
\node at (3,-1){$\bullet$};
\draw[dashed] (4,-1)--(6,-1);
\node  [below] at  (-1.3,-1) {\small $q_3^{-1}x_{i+1}$};
\node  [below] at  (1.5,-1) {\small $p^{-1}q_3^{-1}x_{i+1}$};
\node  [below] at  (4.5,-1) {\small $p^{-2}q_3^{-1}x_{i+1}$};

\node [below] at (0,-3) {\small $x_i$};
\draw [thick] (0,-3)--(0,2);
\draw [thick] (0,2)--(-2,2);
\draw [-latex,thick] (-2,2)--(-2,7);
\end{tikzpicture}
\caption{The pole $q_3^{-1}x_{i+1}$ 
(shown by an open circle) is canceled and the pole $p^{-1}q_1x_{i+1}$ 
(shown by an asterisk) is created. \label{Fig1}}
\end{center}
\end{figure}
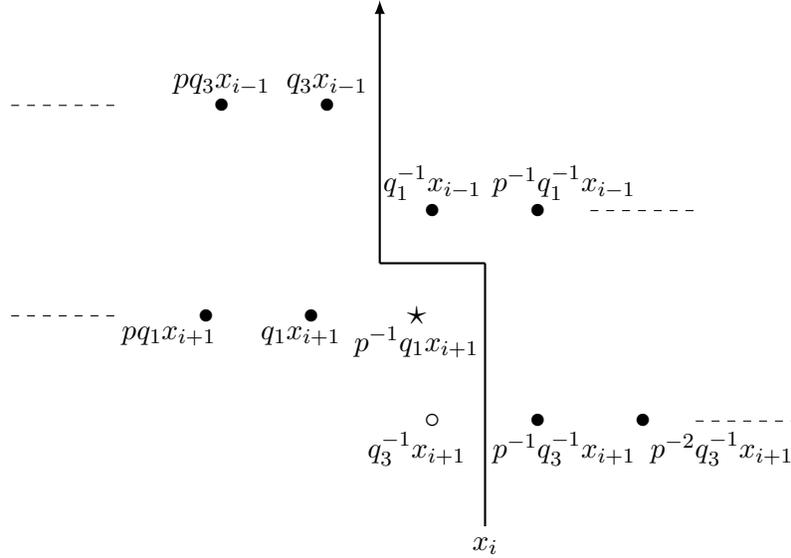

Consider similarly the product $\mbF^{i,n-1}(x_i)\check{\mbF}^{i-1,0}(y_n)$
corresponding to $II$. 
Poles in $y_n$ occur at $y_n=p^{-1}q_1^ix_i$, which corresponds 
to the delta function, and at $y_n=q_1^{i-1}x_{i-1}$ which comes from 
$\mbF^{i-1,0}(x_{i-1})\check{\mbF}^{i-1,0}(y_n)$. 
Upon integration the latter produces a pole at 
$x_i=pq_1^{-1}x_{i-1}$, which should be taken 
outside the contour of $x_i$ integral. 
The pole at $x_i=q_3x_{i-1}$ comes only from 
$\mbF^{i-1,n-1}(x_{i-1})\mbF^{i,n-1}(x_i)$. 
After $y_n$ integration it is canceled by the contraction 
factor $1-q^2\qc_1^nq_1^{-i+1}y_n/x_{i-1}$ coming from 
$\mbF^{i-1,n-1}(x_{i-1})\check{\mbF}^{i,0}(y_n)$. 
See  Figure \ref{Fig2}. 

\begin{figure}[H]
\begin{center}
\begin{tikzpicture}[scale=0.7]
\draw[dashed] (-9,5)--(-7,5);
\node at (-5,5){$\bullet$};
\node at (-3,5){$\circ$};
\node  [above] at (-5.5,5){$pq_3x_{i-1}$};
\node [above] at (-3,5){$q_3x_{i-1}$};
\node at (-3.5,3){$\star$};
\node at (-1,3){$\bullet$};
\node at (1,3){$\bullet$};
\draw[dashed] (2,3)--(4,3);
\node  [above] at  (-3,3) {\small $pq_1^{-1}x_{i-1}$};
\node  [above] at  (-0.5,3) {\small $q_1^{-1}x_{i-1}$};
\node  [above] at  (2.2,3) {\small $p^{-1}q_1^{-1}x_{i-1}$};
\draw[dashed] (-9,1)--(-7,1);
\node at (-5.3,1){$\bullet$};
\node at (-3.3,1){$\bullet$};
\node  [below] at  (-5.3,1) {\small $pq_1x_{i+1}$};
\node  [below] at  (-3.3,1) {\small $q_1x_{i+1}$};
\node at (-1,-1){$\bullet$};
\node at (1,-1){$\bullet$};
\node at (3,-1){$\bullet$};
\draw[dashed] (4,-1)--(6,-1);
\node  [below] at  (-1.3,-1) {\small $q_3^{-1}x_{i+1}$};
\node  [below] at  (1.5,-1) {\small $p^{-1}q_3^{-1}x_{i+1}$};
\node  [below] at  (4.5,-1) {\small $p^{-2}q_3^{-1}x_{i+1}$};

\node [below] at (-2.4,-2.5) {\small $x_i$};
\draw [thick] (-2.4,-2.5)--(-2.4,2);
\draw [thick] (-2.4,2)--(-4.2,2);
\draw [-latex,thick] (-4.2,2)--(-4.2,7);
\end{tikzpicture}
\caption{The contour in the $x_i$ plane. The pole $q_3x_{i-1}$ 
(shown by an open circle) is canceled and the pole $pq_1^{-1}x_{i-1}$ 
(shown by an asterisk) is created. 
\label{Fig2}}

\end{center}
\end{figure}
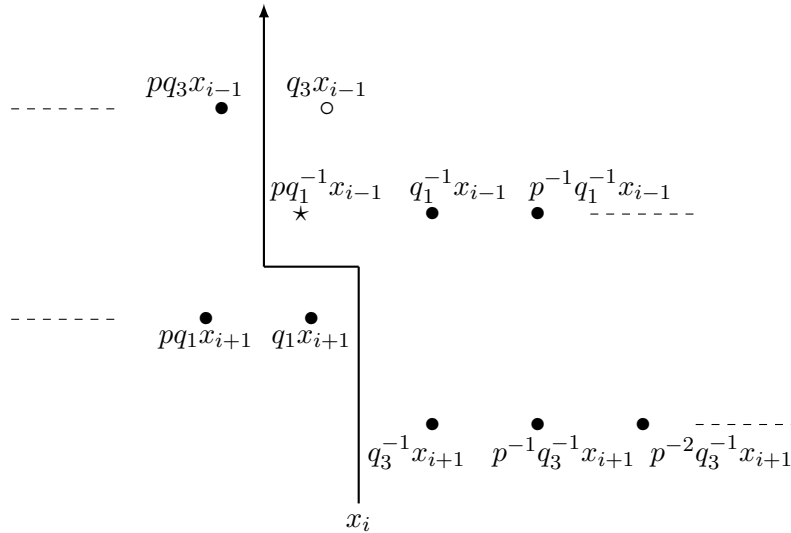

\vskip 1cm

After taking the residue in $y_n$, we perform the shift $x_i\to px_i$ in 
the term $II$. From  Figure \ref{Fig1} and Figure \ref{Fig2} 
the resulting contours become the same. 

In view of the identity \eqref{FbFb}, the term $I$ becomes
\begin{align*}
-\check{u}_ix_i^{-1}\dc^{-n}q
:\mbF^{i,n-1}(px_i)\check{\mbF}^{i-1,0}(q_1^ix_i):
q^{-\varepsilon_{i-1}+\varepsilon_i}\,. 
\end{align*} 

After the shift, the term $II$ becomes 
\begin{align*}
\check{u}_{i-1}x_i^{-1} q^{-n+\delta_{i,1}}
p^{-1}:\mbF^{i,n-1}(px_i)\check{\mbF}^{i-1,0}(q_1^ix_i):\,.
\end{align*}
In addition, the quasi-periodicity \eqref{quasi1} of $h_{\mu,1}$ give rise to  
a factor $p_iq_2^{-\delta_{i,1}}=\bar p_i q^{-2\delta_{i,1}}
q^{-\varepsilon_{i-1}+\varepsilon_i}$, which 
we bring to the position next to $\mbF^{i,n-1}(px_i)\check{\mbF}^{i-1,0}(y_n)$ 
noting
\begin{align*}
\prod_{i+1\le s\le m}^{\curvearrowright}{\mbF}^{s,j_s}(x_{s})\cdot p_i
=q^{1+\delta_{i,1}} \times p_i
\prod_{i+1\le s\le m}^{\curvearrowright}{\mbF}^{s,j_s}(x_{s})\,.
\end{align*}
Putting all these together, 
we conclude that the two terms cancel provided $\bar p_i\uc_{i-1}=\uc_i$. 

In a similar manner, we can show that 
the vanishing holds in the case (ii) of Subsection \ref{FFFF}
provided $\bar \pc_l u_{l-1}= u_l$. 
In the case (iii) the condition for vanishing becomes 
 $\bar p_m\uc_{m-1}=\uc_0$ and  $\bar \pc_n u_{n-1}= u_0$, which is consistent.  

We repeat essentially the same computation for general $M,N$, and for integrals of motion of both kinds.
We omit further details. 
\bigskip

{\bf Acknowledgments.}
The research of BF is supported by 
the Russian Science Foundation grant project 16-11-10316. 
MJ is partially supported by 
JSPS KAKENHI Grant Number JP16K05183. 
EM is partially supported by a grant from the Simons Foundation  
\#353831.

EM and BF would like to thank Kyoto University
for hospitality during their visits when this work was started.


\bigskip

\begin{thebibliography}{0000000}

\bibitem[BLZ]{BLZ1}
V. Bazhanov, S. Lukyanov and A. Zamolodchikov, {\it 
Integrable structure of conformal field theory, quantum KdV theory and 
thermodynamic Bethe ansatz}, 
Commun. Math. Phys.
{\bf {177}} (1996), no.2, 381--398

\bibitem[BLZ1]{BLZ2}
V. Bazhanov, S. Lukyanov and A. Zamolodchikov, 
{\it Integrable structure of conformal field theory II. $Q$-operators
and DDV equation}, 
Commun. Math. Phys.
{\bf 190} (1997), no.2, 247--278

\bibitem[BLZ2]{BLZ3}
V. Bazhanov, S. Lukyanov and A. Zamolodchikov, 
{\it Integrable structure of conformal field theory III. The
Yang-Baxter relation}, 
Commun. Math. Phys.
{\bf 200} (1999), no.2, 297--324

\bibitem[FFR]{FFR} 
B. Feigin, E. Frenkel, and N. Reshetikhin, 
{\it Gaudin model, Bethe ansatz and critical level},  
Commun. Math. Phys. 166 (1994), no.1, 27--62

\bibitem[FJM]{FJM} B. Feigin, M. Jimbo, and E. Mukhin, 
{\it Integrals of motion from quantum toroidal algebras}, 
{\em J.Phys.A: Math. Theor.} {\bf 50} (2017) 464001 

\bibitem[FJMM]{FJMM} B. Feigin, M. Jimbo, T. Miwa, and E. Mukhin, 
{\it Branching rules for quantum toroidal $\mathfrak{gl}_N$}, 
 Adv. Math. \textbf{300} (2016) 229--274


\bibitem[FJMM1]{FJMM1} B. Feigin, M. Jimbo, T. Miwa, and E. Mukhin, 
{\it Quantum toroidal $\mathfrak{gl}_1$
and Bethe ansatz}, J.Phys.A: Math. Theor. \textbf{48} (2015) 244001

\bibitem[FJMM2]{FJMM2} B. Feigin, M. Jimbo, T. Miwa, and E. Mukhin, 
{\it Finite type modules and  Bethe ansatz for the quantum 
toroidal $\mathfrak{gl}_1$}, Commun. Math. Phys., \textbf{356},  (2017), no.1, 285–-327

\bibitem[FJMM3]{FJMM3} B. Feigin, M. Jimbo, T. Miwa, and E. Mukhin, {\it Representations of quantum toroidal $\gl_n$}, J. Algebra {\bf 380} (2013), 78–-108

\bibitem[FKSW]{FKSW} B. Feigin, T. Kojima, J. Shiraishi and 
H. Watanabe, 
{\it The integrals of motion for the deformed Virasoro algebra},
arXiv:0705.0427v2

\bibitem[FKSW1]{FKSW1} B. Feigin, T. Kojima, J. Shiraishi and H. Watanabe,
{\it The integrals of motion for the deformed $W$ algebra
$W_{q,t}\bigl(\widehat{\mathfrak{sl}}_N\bigr)$}, 
arXiv:0705.0627v1

\bibitem[KS]{KS} T. Kojima and J. Shiraishi,
{\it The integrals of motion for the deformed $W$ algebra
$W_{q,t}\bigl(\widehat{\mathfrak{sl}}_N\bigr)$ II:
Proof of the commutation relations}, 
 Commun. Math. Phys. \textbf{283} (2008), no.3, 795--851 

\bibitem[Mi]{Mi} K. Miki, {\it Toroidal braid group action and an automorphism
of toroidal algebra $U_q\bigl(\mathfrak{sl}_{n+1,tor}\bigr)$ ($n\ge2$)},
{Lett. Math. Phys.} {\bf 47} (1999), no.4,  365--378

\bibitem[MTV]{MTV1} E. Mukhin, V. Tarasov and A. Varchenko, 
{\it A Generalization of the Capelli Identity},
Algebra, arithmetic, and geometry: in honor of Yu. I. Manin. Vol. II, 
Progr. Math., 270 (2009),  Birkhauser Boston, Inc., Boston, MA, 383--398


\bibitem[MTV1]{MTV2} E. Mukhin, V. Tarasov and A. Varchenko, 
{\it Bispectral and $(\gl_N,\gl_M)$ dualities},
 Funct. Anal. Other Math. {\bf 1} (2006), no.1, 47--69


\bibitem[MTV2]{MTV3} E. Mukhin, V. Tarasov and A. Varchenko, 
{\it Bispectral and $(\gl_N,\gl_M)$ dualities, discrete versus differential},
Adv. Math. {\bf 218} (2008) no.1, 215--265



\bibitem[Sa]{Sa} Y. Saito, {\it Quantum toroidal algebras and their vertex representations},
 {Publ. RIMS, Kyoto Univ.} \textbf{34} (1998), no.2, 155--177

\bibitem[STU]{STU} Y. Saito, K. Takemura, and D. Uglov, {\it Toroidal actions on level 1 modules for $U_q(\hat{\mathfrak{sl}}_n)$}, Transform. Groups \textbf{3} (1998), 
no. 1, 75--102


\end{thebibliography}
\end{document}